\documentclass[12pt]{amsart}
\usepackage{amssymb,latexsym}
\usepackage{enumerate}

\makeatletter
\@namedef{subjclassname@2010}{%
  \textup{2010} Mathematics Subject Classification}
\makeatother

\newtheorem{thm}{Theorem}[section]
\newtheorem{cor}[thm]{Corollary}
\newtheorem{lem}[thm]{Lemma}

\newtheorem{prop}[thm]{Proposition}

\theoremstyle{definition}

\newtheorem{rem}[thm]{Remark}

\numberwithin{equation}{section}
\usepackage[all]{xy}
\usepackage{url}
\newcommand{\Z}{{\mathbb{Z}_+}}
\newcommand{\N}{\mathbb{N}}
\newcommand{\F}{\mathcal F}

\newcommand{\bN}{\beta\mathbb{N}}
\newcommand{\bZ}{\beta\mathbb{Z}_+}
\newcommand{\mZ}{\mathbb Z}

\newcommand{\pn}{\mathcal{P}_f(\N)}
\newcommand{\ff}{Force(\F)}
\begin{document}
\title[Dynamical characterization of C-sets and its application]{Dynamical characterization of C-sets and its application}
\author[J. Li]{Jian Li}
\address{Department of Mathematics\\ University of Science and Technology of China\\
Hefei, Anhui, 230026, P.R. China}
\email{lijian09@mail.ustc.edu.cn}
\date{}

\begin{abstract}
In this paper, we set up a general correspondence between the algebra properties of $\bN$
and the sets defined by dynamical properties.
In particular, we obtain a dynamical characterization of C-sets,
where C-sets are the sets satisfying the strong Central Sets Theorem.
As an application, we show that Rado systems are solvable in C-sets.
\end{abstract}

\subjclass[2010]{Primary 37B20; Secondary 37B05, 05D10}

\keywords{Central sets, C-sets, Rado system}

\maketitle

\section{Introduction}
Throughout this paper, let $\mathbb Z$, $\mathbb Z_+$, $\mathbb N$ and $\mathbb Q$
denote the sets of the integers, the non-negative integers,
the positive integers and the rational numbers, respectively.
Let us recall two celebrated theorems in combinatorial number theory.

\begin{thm}[van de Waerden]
Let $r\in\N$ and $\N=\bigcup_{i=1}^r C_i$. Then there exists
some $i\in\{1,2,\ldots, r\}$ such that
$C_i$ contains arbitrarily long arithmetic progressions.
\end{thm}

For a sequence $\{x_n\}_{n=1}^{\infty}$ in $\N$,
define the finite sums of $\{x_n\}_{n=1}^{\infty}$ as
$$FS(\{x_n\}_{n=1}^{\infty})=\left\{ \sum\nolimits_{n\in \alpha}x_n:\,
\alpha \textrm{ is a nonempty finite subset of }\N\,\right\}.$$
A subset $F$ of $\N$ is called an \emph{IP set}
if there exists a sequence $\{x_n\}_{n=1}^{\infty}$ in $\N$
such that $ FS(\{x_n\}_{n=1}^{\infty})\subset F$.

\begin{thm}[Hindman]
Let $r\in\N$ and $\N=\bigcup_{i=1}^r C_i$. Then there exists
some $i\in\{1,2,\ldots, r\}$ such that $C_i$ is an IP set.
\end{thm}

The original proofs of the above two theorems are somewhat complicated by combinatorial methods.
In \cite{FW78,F81} Furstenberg and Weiss found a new way to prove those theorems
by topological dynamics methods.

A subset $F$ of $\N$ is called \emph{central} if there exists a dynamical system $(X,T)$,
a point $x \in X$, a minimal point $y$ which is proximal to $x$, and an open neighborhood $U$
of $y$ such that $F=\{n\in\N: T^nx\in U\}$.
Then van de Waerden Theorem and Hindman Theorem follow from the following result.
\begin{thm}[\cite{FW78, F81}]
(1) Every central set is an IP set and contains arbitrarily long arithmetic progressions.

(2) Let $r\in\N$ and $\N=\bigcup_{i=1}^r C_i$. Then there exists
some $i\in\{1,2,\ldots, r\}$ such that $C_i$ is a central set.
\end{thm}

Before going on, let us recall some notions.
We call $(S, \cdot)$ a \emph{compact Hausdorff right topological semigroup}
if $S$ is endowed with a topology with respect to which $S$ is a compact
Hausdorff space and for each $t\in S$ the right translation $s\mapsto s\cdot t$ is continuous.
An \emph{idempotent} $t\in S$ is an element satisfying $t\cdot t = t$.
Ellis-Namakura Theorem says that any compact Hausdorff right topological semigroup contains some idempotent.
A subset $I$ of $S$ is called a \emph{left ideal of $S$} if $SI\subset I$,  a \emph{right ideal} if $IS\subset I$,
and a \emph{two sided ideal} (or simply an \emph{ideal}\/) if it is both a left and  right ideal.
A \emph{minimal left ideal} is the left ideal that does not contain any proper left ideal.
Similarly, we can define \emph{minimal right ideal} and \emph{minimal ideal}.
An idempotent in $S$ is called a \emph{minimal idempotent} if it is contained in some minimal left ideal of $S$.

Endowing $\N$ with the discrete topology, we take the points of the Stone-\v{C}ech compactification
$\bN$ of $\N$ to be the ultrafilter on $\N$.
Since $(\N,+)$ is a semigroup, we extend the operation $+$ to $\bN$ such that
$(\bN, +)$ is a compact Hausdorff right topological semigroup.
See \cite{HS98} for an exhaustive treatment of the algebraic structure on $\bN$.

Ellis showed that we can regard $(\bN,\N)$ as a universal point transitive system (\cite{E69}).
One may think that there is a nature connection between the algebra properties of $\bN$
and the sets defined by dynamical properties.
For example, in \cite{BH90} Bergelson and Hindman showed that
\begin{thm}[\cite{BH90}]
A subset $F$ of $\N$ is central if and only if there exists a minimal idempotent $p\in\bN$ such that $F\in p$.
\end{thm}

A subset $F$ of $\N$ is  called  \emph{quasi-central} if there exists an idempotent $p\in\bN$ with
each element being piecewise syndetic such that $F\in p$.
Of course, every quasi-central set is central, but there exists some quasi-central set which is not central (\cite{HMS96}).
The authors in \cite{BH07} showed a dynamical characterization of quasi-central sets.

\begin{thm}[\cite{BH07}] \label{thm:quasi-cen-dy}
A subset $F$ of $\N$ is quasi-central if and only if
there exists a dynamical system $(X,T)$, a pair of points $x, y\in X$ where
for every open neighborhood $V$ of $y$ the set
$\{n\in\N: T^nx\in V, T^ny\in V\}$ is piecewise syndetic,
and an open neighborhood $U$ of $y$
such that $F=\{n\in \N: T^n x\in U\}$.
\end{thm}

A subset $F$ of $\N$ is  called a \emph{D-set} if there exists an idempotent $p\in\bN$ with
each element having positive upper Banach density such that $F\in p$.
It should be noticed that every quasi-central set is a D-set and
there exists some D-set which is not quasi-central (\cite{BD08}).
There is also a dynamical characterization of D-sets.

\begin{thm}[\cite{BD08}] \label{thm:D-set-dyn}
A subset $F$ of $\N$ is a D-set if and only if
there exists a dynamical system $(X,T)$, a pair of  points $x,y\in X$ where for every open neighborhood $V$ of $y$ the set
$\{n\in\N: T^ny\in V\}$ has positive upper Banach density and  $(y,y)$ belongs the
orbit closure of $(x,y)$ in the product system $(X\times X, T\times T)$,
and an open neighborhood $U$ of $y$
such that $F=\{n\in \N: T^n x\in U\}$,
\end{thm}

Central sets have substantial combinatorial contents.
In order to describe the properties, we first introduce some notations.
By $\pn$ we denote the set of all nonempty
finite subsets of $\N$. For $\alpha, \beta\in\pn$, we  write $\alpha<\beta$ if $\max \alpha<\min \beta$.
Given a sequence $s_1,s_2,\ldots$ in $\mZ$ or $\mZ^m$ and $\alpha\in\pn$ we let
$s_\alpha=\sum_{n\in\alpha}s_n$ and call the family $(s_\alpha)_{\alpha\in\pn}$ an \emph{IP-system}.
A \emph{homomorphism} $\phi: \pn\to\pn$ is a map such that (1) if $\alpha\cap \beta=\emptyset$, then
$\phi(\alpha)\cap\phi(\beta)=\emptyset$ and (2) $\phi(\alpha\cup\beta)=\phi(\alpha)\cup\phi(\beta)$.
Evidently such a homomorphism is determined by $\phi(\{i\})$ on each $i\in\N$, and then
$\phi(\alpha)=\bigcup_{i\in\alpha}\phi(\{i\})$.
Given an IP-system $\{s_\alpha\}$, an IP-subsystem is defined by a homomorphism $\phi: \pn\to\pn$
and forming $\{s_{\phi(\alpha)}\}\subset\{s_\alpha\}$.
If $r\in\mZ$, we shall denote by $\bar r^{(m)}$ the vector $(r,\ldots,r)\in\mZ^m$.

\begin{prop}[Central Sets Theorem \cite{F81}]
Let $F$ be a central set in $\N$, and for any $m\geq 1$, let $\{s_\alpha\}$ be any IP-system in $\mZ^m$.
We can find an IP-subsystem $\{s_{\phi(\alpha)}\}$ and an IP-system $\{r_\alpha\}$ in $\N$
such that the vector $\bar r^{(m)}_{\alpha}+s_{\phi(\alpha)}\in F^m$ for each $\alpha\in\pn$.
\end{prop}

Recently, the authors in \cite{DHS08,HS09} proved a stronger version of the Central Sets Theorem,
they call C-sets are the sets satisfying the conclusion of the strong Central Sets Theorem.
Here we will not discuss the strong Central Sets Theorem,
so we adopt an alternative definition of C-sets.

A subset $F$ of $\N$ is called a \emph{J-set} if for every $m\in\N$ and
every IP-system $\{s_\alpha\}$ in $\mZ^m$
there exists $r\in\N$ and $\alpha\in\pn$ such that $\bar r^{(m)}+s_\alpha \in F^m$.
Denote by $\mathcal J$ the collection of all J-sets.
A subset $F$ of $\N$ is called a \emph{C-set} if there exists an idempotent
$p\in\bN$ with each element being J-set such that $F\in p$.
Since every positive upper Banach density set is a J-set (\cite{FK85}), every D-set is a C-set.
But there exists a C-set with zero upper Banach density (\cite{H09}), so this set is not a D-set.

In this paper, we obtain a dynamical characterization of C-sets.

\begin{thm}
A subset $F$ of $\N$ is a C-set if and only if there exists a dynamical system $(X, T)$,
a pair of points $x, y\in X$ where for any open neighborhood $V$ of $y$ the set $\{n\in\N: T^n y\in V\}$ is a J-set
and $(y,y)$ belongs to the orbit closure of $(x,y)$ in the product system $(X\times X,T\times T)$,
and an open neighborhood $U$ of $y$ such that $F=\{n\in\N: T^nx\in U\}$,
\end{thm}

In \cite{F81} Furstenberg used the Central Sets Theorem to show that
any central subset of $\N$ contains solutions to all Rado systems.
Let $A=(a_{ij})$ be a $p\times q$ matrix over  $\mathbb Q$,
the homogeneous system of linear equations
\[A(x_1,\ldots,x_q)^T=0\]
is called \emph{partition regular} if for every $r\in\N$ and
$\N=\bigcup_{i=1}^r C_i$, there exists some $i\in\{1,2,\ldots, r\}$
such that the system has a solution
$(x_1,\ldots,x_q)$ all of whose components lie in $C_i$.
Each such homogeneous system of linear equations is called a \emph{Rado system}.
In \cite{R33} Rado characterized when a homogeneous system of linear equations is partition regular.

\begin{thm}[Rado's Theorem]
Let $A=(a_{ij})$ be a $p\times q$ matrix over  $\mathbb Q$.
Then the system $A(x_1,\ldots,x_q)^T=0$ is  partition regular
if and only if the index set $\{1,2,\ldots,q\}$ can be divided into $l$ disjoint subsets
$I_1,I_2,\ldots,I_l$ and rational numbers $c_j^r$ may be found for $r\in \{1,\ldots,l\}$ and
$j\in I_1\cup\cdots\cup I_r$ such that the following relationships are satisfied:
\begin{align*}
&\sum_{j\in I_1}a_{ij}=0,\\
&\sum_{j\in I_2}a_{ij}=\sum_{j\in I_1}c^1_ja_{ij},\\
&\cdots\\
&\sum_{j\in I_l}a_{ij}=\sum_{j\in I_1\cup I_2\cup\cdots\cup I_{l-1}}c^{l-1}_{j}a_{ij}.\\
\end{align*}
\end{thm}

Let $F$ be a subset of $\N$, we say Rado systems are \emph{solvable} in $F$
if every Rado system $A(x_1,\ldots,x_q)^T=0$ has a solution
$(x_1,\ldots,x_q)$ all of whose components lie in $F$.

Furstenberg and Weiss improved Rado's result by showing that
\begin{thm}\cite{FW78,F81}
Rado systems are solvable in central sets.
\end{thm}

Recently, the authors in \cite{BBDF09} extended Furstenberg and Weiss' result to
\begin{thm} \cite{BBDF09}
Rado systems are solvable in D-sets.
\end{thm}

In this paper, we use the dynamical characterization of C-sets to show that
\begin{thm}
Rado systems are solvable in C-sets.
\end{thm}

This paper is organized as follows.
In Section 2 we introduce some notations related to Furstenberg families.
In Section 3 the basic properties of the Stone-\v{C}ech compactification of $\N$ are discussed.
In Section 4 we set up a general correspondence between the algebra properties of $\bN$
and the sets defined by dynamical properties.
The dynamical characterizations of quasi-central sets and D-sets
are special cases of our results.
In Section 5,  we investigate the set's forcing
that is the dynamical properties of a point along a subset of $\N$.
In Section 6, we consider both addition and multiplication in N and $\bN$.
Particularly we show that if $F$ is a quasi-central set or a D-set,
then for every $n\in\N$ both $nF$ and $n^{-1}F$ are also quasi-central sets or D-sets.
In Section 7 using the correspondence which is set up in Section 4 and
some properties of J-sets, we obtain a dynamical characterization of C-sets.
In Section 8, as an application, we give a topological dynamical proof of
the fact that Rado systems are solvable in C-sets.

\section{Furstenberg family}
Let us recall some notations related to a family (for more details see \cite{A97}).
For the set of positive integers $\N$,
denote by $\mathcal{P}=\mathcal{P}(\N)$ the collection of all subsets of $\N$.
A subset $\mathcal{F}$ of $\mathcal{P}$ is called a \emph{Furstenberg family} (or just \emph{family})
if it is hereditary upward, i.e., $F_1\subset F_2$ and $F_1\in\mathcal F$ imply $F_2\in \mathcal F$.
A family $\mathcal F$ is called \emph{proper} if it is a nonempty proper subset of $\mathcal P$,
i.e., neither empty nor all of $\mathcal P$.
For a family $\F$, the \emph{dual family} of $\mathcal F$, denote by $\kappa\F$, is
\[\{F\in \mathcal P: F\cap F'\neq \emptyset, \forall F'\in \F\}.\]
Sometimes the dual family $\kappa\F$ is also denoted by $\F^*$.

A family $\F$ is called a \emph{filter} when it is a proper family closed under intersection,
i.e., if $F_1,F_2\in \F$ then $F_1\cap F_2\in \F$.
A family $\F$ is called a \emph{filterdual} if its dual $\kappa\F$ is a filter.
It is easy to see that a proper family $\F$ is a filterdual if and only if it satisfies the \emph{Ramsey Property},
i.e., if $F_1\cup F_2\in\F$ then either $F_1\in\F$ or $F_2\in\F$.
Since $\kappa(\kappa\F)=\F$, a family $\F$ is a filter if and only if $\kappa\F$ is a filterdual.

Of special interest are filter that are maximal with respect to inclusion.
Such a filter is called an \emph{ultrafilter}. By Zorn's Lemma every filter is contained in some ultrafilter.
For any $n\in\N$ the family $\{A\subset\N: n\in A\}$ is an ultrafilter.
An ultrafilter formed in this way is called a \emph{principal ultrafilter}.
Any other ultrafilter is called a \emph{non-principal ultrafilter}.
The following two lemmas are basic properties of ultrafilter, see \cite{A97,G80, HS98} for example.

\begin{lem}
Let $\F$ be a filter. Then the following conditions are equivalent:
\begin{enumerate}
\item $\F$ is an ultrafilter;
\item $\F=\kappa\F$;
\item $\F$ is a filterdual;
\item For all $F\subset\N$, either $F\in\F$ or $\N\setminus F\in\F$.
\end{enumerate}
\end{lem}

\begin{lem}\label{lem:filter-dual-ultra}
Let $\F$ be a filterdual and $\mathcal A\subset \F$.
If for any finite elements $A_1, A_2$,$\ldots$, $A_n$ in $\mathcal A$ the intersection
$\bigcap_{i=1}^n A_i$ is  in $\F$, then there exists an ultrafilter $\F'$ such that
$\mathcal A\subset \F'\subset \F$.
\end{lem}

For $n\in \mZ$ and $F\subset \N$, denote $n+F=\{n+m\in \N: m\in F\}$.
A family $\F$ is called \emph{translation $+$ invariant} if $n+F\in \F$ for every $n\in\Z$ and $F\in\F$,
\emph{translation $-$ invariant} if $-n+F\in\F$ for every $n\in\Z$ and $F\in\F$ and
\emph{translation invariant} if it is both $+$ and $-$ invariant.

Any nonempty collection $\mathcal A$ of
subsets of $\N$ naturally generates a family
\[\F(\mathcal A)=\{F\subset \N:\, F\supset A\textrm{ for some }A \in\mathcal A\}.\]
A collection $\mathcal A$ of subsets of $\N$ is said to have the \emph{finite intersection property}
if the intersection of any finite elements in $\mathcal A$ is not empty.
In this case, the family generated by $\mathcal A$ is a filter.

Let $\mathcal F$ be a family, the \emph{block family of $\F$}, denote by $b\F$, is the
family consisting of sets $F\subset\N$ for which there exists some $F'\in \F$ such that for
every finite subset $W$ of $F'$ one has $m+W \subset F$ for some $m\in \Z$.
It is easy to see that $F\in b\F$ if and only if there exists a sequence $\{a_n\}_{n=1}^\infty$ in $\Z$
and $F'\in\F$ such that $\bigcup_{n=1}^\infty (a_n+F'\cap[1,n])\subset F$.
Clearly, $b(b\F)=b\F$ and $b\F$ is translation $+$ invariant.

\begin{lem} \cite{BF02,HLY09}
If $\F$ is a filterdual, then so is $b\F$.
\end{lem}

Now let us recall some important sets and families.
Let $\mathcal F_{inf}$ be the family of all infinite subsets of $\mathbb Z_+$.
It is easy to see that its dual family $\kappa\mathcal F_{inf}$
is the family of all cofinite subsets, denoted by $\mathcal F_{cf}$.

A subset $F$ of $\Z$ is called \emph{thick} if it contains arbitrarily long runs of positive integers,
i.e., there exists a sequence $\{a_n\}_{n=1}^\infty$ in $\Z$ such that $\bigcup_{n=1}^\infty (a_n+[1,n])\subset F$;
\emph{syndetic} if there exists $N\in\mathbb N$ such that $[n,n+N]\cap F\neq \emptyset$ for every $n\in\N$;
\emph{piecewise syndetic} if it is the intersection of a thick set and a syndetic set.
The families of all thick sets, syndetic sets and piecewise syndetic sets
are denoted by $\mathcal F_{t}$, $\mathcal F_{s}$ and $\mathcal F_{ps}$, respectively.
It is easy to see that $\kappa \mathcal F_s=\mathcal F_t$.

Let $F$ be a subset of $\N$, the \emph{upper density of $F$} is
$$\bar{d}(F)=\limsup_{n\to\infty}\frac{|F\cap [1,n]|}{n},$$
where $|\cdot|$ denote the cardinality of the set,
and the \emph{upper Banach density of $F$} is
$$BD^*(F)=\limsup_{|I|\to\infty}\frac{|F\cap I|}{|I|}$$
where $I$ is taken over all nonempty finite intervals of $\N$.
Using density we can define lots of interesting families. For example,
denote $\F_{pud}$ and $\F_{pubd}$ by the family of sets with
positive upper density and positive upper Banach density respectively.

Denote by $\F_{ip}$ and $\F_{cen}$ the family of all IP sets and central sets respectively.
We have the following basic property about the familiar families, see \cite{A97,HS98} for example.

\begin{lem}
\begin{enumerate}
\item $\F_{cen}$, $\F_{ip}$, $\F_{ps}$, $\F_{pud}$ and $\F_{pubd}$ are filterduals.
\item $\F_{ps}$, $\F_{pud}$, $\F_{pubd}$ and $\F_{s}$ are translation invariant.
\item $b\F_{cf}=\F_t$, $b\F_{s}=\F_{ps}$ and $b\F_{pud}=\F_{pubd}$.
\end{enumerate}
\end{lem}

We now introduce the notion of $\F$-limit.
Let $\F$ be a family and $\{x_n\}_{n\in\N}$  be a sequence in a topological space,
we say that $x$ is a $\F$-limit of $\{x_n\}$ if for every open neighborhood $U$ of $x$
the set $\{n\in\N: x_n\in U\}\in\F$.
Then $\F_{cf}$-limit is just the ordinary convergence.
It is easy to check that If $\F$ is a filter then $\F$-$\lim x_n$ exists and is unique in
every compact Hausdorff space.

\section{$\bN$: the Stone-\v{C}ech compactification of $\N$}
Endowing $\N$ with the discrete topology, we take the points of the Stone-\v{C}ech compactification
$\bN$ of $\N$ to be the ultrafilter on $\N$,
the principal ultrafilter being identified with the points in $\N$.
For $A\subset\N$, let $\overline{A}=\{p\in\bN: A\in p\}$.
Then the sets $\{\overline{A}: A\subset\N \}$ forms a
basis for the open sets (and a basis for the closed sets) of $\bN$.

Since $(\N,+)$ is a semigroup, we can extend the operation $+$ to $\bN$ as
\[p+q=\{F\subset \N: \{n\in\N: -n+F\in q\}\in p\}.\]
Then $(\bN, +)$ is a compact Hausdorff right topological semigroup
with $\N$ contained in the topological center of $\bN$.
That is, for each $p\in \bN$ the map $\rho_p: \bN\to\bN$, $q\mapsto q+p$ is continuous,
and for each $n\in\N$ the map $\lambda_n: \bN\to\bN$, $q\mapsto n+q$ is continuous.
It is well known that $\bN$ has a smallest ideal
$K(\bN)=\bigcup\{L: L $ is a minimal left ideal of $\bN\}
= \bigcup\{R: R $ is a minimal right ideal of $\bN\}$ (\cite[Theorem 2.8]{HS98}).

\begin{lem}
Let $\F$ be a filter. If for every $F\in \F$ there exists some $F'\in F$ such that
$-n+F\in \F$ for every $n\in F'$, then $\bigcap_{F\in\F}\overline{F}$ is a closed subsemigroup of $\bN$.
\end{lem}
\begin{proof}
Since $\F$ has finite intersection property, $\bigcap_{F\in\F}\overline{F}$ is nonempty.
Let $p,q\in \bigcap_{F\in\F}\overline{F}$, we want to show that $p+q\in \bigcap_{F\in\F}\overline{F}$.
Let $F\in \F$, it suffices to show that $F\in p+q$.
For this $F$, there exists some $F'\in \F$ such that $-n+F\in \F$ for every $n\in F'$.
Then $F'\subset \{n\in \N: -n+F\in q\}$ and $\{n\in\N: -n+F\in q\}\in p$.
By the definition of ``$+$" in $\bN$ we have $F\in p+q$.
\end{proof}

\begin{lem} [\mbox{\cite[Theorem 4.20]{HS98}}] \label{lem:col-semigroup}
Let $\mathcal A$ be a collection of subset of $\N$. If $\mathcal A$ has the finite intersection property and
for every $F\in\mathcal A$ and $n\in F$ there exists $F'\in \mathcal A$ such that $n+F'\subset F$, then
 $\bigcap_{F\in\mathcal A}\overline{F}$ is a closed subsemigroup of $\bN$.
\end{lem}

For a filterdual $\F$, the \emph{hull} of $\F$ is defined by
\[h(\F)=\{p\in\bN: p\subset\F\}.\]
Then $h(\F)$ is a nonempty closed subset of $\bN$ and $F\in \F$ if and only if
$\overline{F}\cap h(\F)\neq\emptyset$.
Conversely, for a nonempty closed subset $Z$ of $\bN$, the \emph{kernel} of $Z$ is defined by
\[k(Z)=\{F\subset \N: \overline{F}\cap Z\neq\emptyset \}.\]
Then $k(Z)$ is a filterdual, $h(k(Z))=Z$ and $k(h(\F))=\F$.
In this means, we obtain a one-to-one corresponding between the set of filterduals on $\N$
and the set of nonempty closed subsets of $\bN$ (\cite{E69,G80}).
\begin{lem} \cite{G80,HS98} We have the following correspondences.
\begin{enumerate}
\item $h(\F_{ps})=\overline{K(\bN)}$.
\item $h(\F_{ip})=\overline{\{p\in\bN: p \textrm{ is an idempotent}\}}$.
\item $h(\F_{cen})=\overline{\{p\in\bN: p \textrm{ is a minimal idempotent}\}}$.
\item $h(\F_{pubd})=\overline{\bigcup\{\mathrm{Supp}(\mu): \mu\in \mathcal M\}}$, where
$\mathcal M$ is the set of all $\N$-invariant probability measure on $\bN$.
\end{enumerate}
\end{lem}

\begin{lem}\label{lem:filterdual-tran-invar}
Let $\F$ be a filterdual. Then $\F$ is translation $+$ invariant if and only if
$h(\F)$ is a closed left ideal of $\bN$.
\end{lem}
\begin{proof}
Assume that $\F$ is translation $+$ invariant. In order to show that $h(\F)$ is a closed left ideal,
it suffices to show that $m+h(\F)\subset h(\F)$ for every $m\in \N$.
Let $m\in\N$, $p\in h(\F)$ and $F\in m+p$. Then $m\in \{n\in \N: -n+F\in p\}$ and $-m+F\in p\subset \F$.
Since $\F$ is translation $+$ invariant, $m+(-m+F)\subset F$, then $F\in \F$ and
$m+p\subset\F$, i.e., $m+p\in h(\F)$.

Conversely, assume that $h(\F)$ is a closed left ideal of $\bN$. Let $F\in\F$ and $n\in\N$,
we want to show that $n+F\in\F$.
By Lemma \ref{lem:filter-dual-ultra}, there exists some $p\in h(\F)$ with $F\in p$.
Clearly, $n\in \{m\in\N: -m+(n+F)\in p\}$, so $n+F\in n+p\in h(\F)$ and $n+F\in\F$.
\end{proof}

\begin{lem}
Let $\F$ be a filterdual and $b\F=\F$. Then $h(\F)$ is a closed two sided ideal of $\bN$.
\end{lem}
\begin{proof}
Since $b\F$ is translation $+$ invariant, by Lemma \ref{lem:filterdual-tran-invar},
$h(\F)$ is a closed left ideal of $\bN$. Then it suffices to show that
$h(\F)$ is also a right ideal.

Let $p\in h(\F)$, $q\in \bN$ and $A\in p+q$, we need to show that $A\in\F$.
Let $F=\{n\in\N: -n+A\in q\}$. Then $F\in p\subset \F$.
For every finite subset $E$ of $F$, $\bigcap_{n\in E}(-n+A)\in q$ is not empty,
choose $n_E\in \bigcap_{n\in E}(-n+A)$, then $n_E+E\subset A$. This implies $A\in b\F=\F$.
\end{proof}

Let $\F$ be a filterdual, we call $F\subset \N$ an \emph{essential $\F$-set}
if there is an idempotent $p \in h(\F)$ such that $F\in p$.
Denote by $\widetilde{\F}$ the collection of all essential $\F$-sets.
Then $\widetilde{\F}$ is also a filterdual since
\[h(\widetilde{\F})= \overline{\{p\in\bN: p \textrm{ is an idempotent in } h(\F)\}}.\]

We have the following observations:

Let $F$ be a subset of $\N$. Then
\begin{enumerate}
\item $F$ is an IP set if and only if it is an essential $b\F_{ip}$-set.
\item $F$ is a quasi-central set if and only if it is an essential $\F_{ps}$-set.
\item $F$ is a D-set if and only if it is an essential $\F_{pubd}$-set.
\item $F$ is a C-set if and only if it is an essential $\mathcal J$-set, where $\mathcal J$ is the
collection of all J-sets.
\end{enumerate}

\begin{thm}
Let $\F$ be a translation invariant filterdual and $\{x_n\}_{n=1}^\infty$ be a sequence in $\N$.
If $FS(\{x_n\}_{n=1}^\infty)\in \F$, then
for every $m\in\N$, $FS(\{x_n\}_{n=m}^\infty)$ is an essential $\F$-set.
\end{thm}
\begin{proof}
We first prove the following claim.

{\bf Claim}: For each $m\in\N$, $h(\F)\bigcap \overline{FS(\{x_n\}_{n=m}^\infty)}\neq\emptyset$.

{\bf Proof of the Claim}: Clearly, the claim holds for $m=1$. Now assume that $m\geq 2$, then
\begin{align*}
 FS(\{x_n\}_{n=1}^\infty)= &FS(\{x_n\}_{n=1}^{m-1})\bigcup FS(\{x_n\}_{n=m}^\infty) \\
 & \bigcup \left\{t+FS(\{x_n\}_{n=m}^\infty): t\in  FS(\{x_n\}_{n=1}^{m-1})\right\}.
\end{align*}

Since $\F$ is translation invariant, $p$ cannot be a principle ultrafilter, then
the finite set $FS(\{x_n\}_{n=1}^{m-1})$ is not in $p$.
If $FS(\{x_n\}_{n=m}^\infty)\in p$, then the claim holds.
Now assume that we have some $t\in FS(\{x_n\}_{n=1}^{m-1})$ such that $t+FS(\{x_n\}_{n=m}^\infty)\in p$.
Choose $q\in \overline{FS(\{x_n\}_{n=m}^\infty)}$ such that $t+q=p$. For every $F\in q$,
$t\in \{n\in \N: -n+(t+F)\in q\}$, so $t+F\in p\subset \F$.
Since $\F$ is translation invariant, we have $F\in \F$ and $q\in h(\F)$.
This ends the proof of the claim.

By Lemma \ref{lem:col-semigroup} $\bigcap_{m=1}^\infty \overline{FS(\{x_n\}_{n=m}^\infty)}$
is a closed subsemigroup of $\bN$, and
by Lemma \ref{lem:filterdual-tran-invar} $h(\F)$ is a closed left ideal of $\bN$.
Then by the above claim we have $h(\F)\bigcap \bigcap_{m=1}^\infty \overline{FS(\{x_n\}_{n=m}^\infty)}$
is a nonempty subsemigroup of $\bN$. By the well known Ellis-Namakura Theorem, there exists some
idempotent in $h(\F)\bigcap \bigcap_{m=1}^\infty \overline{FS(\{x_n\}_{n=m}^\infty)}$.
Thus  for every $m\in\N$, $FS(\{x_n\}_{n=m}^\infty)$ is an essential $\F$-sets.
\end{proof}

For convenience, we also consider $\bZ$ the Stone-\v Cech compactification of $\Z$.
There is a nature imbedding map $i:\bN\to \bZ$ by $i(p)=p\bigcup \{A\cup\{0\}: A\in p\}$.
Then we can regard $\bN$ as a subset of $\bZ$ and $\bZ=\bN\cup\{0\}$.
The advantage of $\bZ$ is that it contains the identity element $0$,
but we don't want to take $0$ into account when considering the multiplication.

\section{Relationships between algebra properties of $\bN$ and dynamical properties}

A \emph{topological dynamical system} (or just \emph{system}) is a pair $(X, T)$,
where $X$ is a nonempty compact Hausdorff space and $T$ is a continuous map from $X$ to itself.
When $X$ is metrizable or $T$ is a homeomorphism,
we call $(X,T)$ a \emph{metrizable} or \emph{invertible dynamical system} respectively.

Let $(X,T)$ be a dynamical system and $x\in X$,
denote the \emph{orbit} of $x$ by $Orb(x, T)=\{T^nx: n\in\Z\}$.
Let $\omega(x, T)$ be the \emph{$\omega$-limit set} of $x$,
i.e., $\omega(x,T)$ is the limit set of $Orb(x,T)$.
A point $x\in X$ is called a \emph{recurrent point} if $x\in\omega(x,T)$.
We call a system $(X,T)$ is \emph{minimal} if it contains no proper subsystems,
and $x\in X$ is a \emph{minimal point} if it belongs to some minimal subsystem of $X$.

A \emph{factor map} $\pi: (X,T)\to (Y,S)$ is a continuous surjective
map from $X$ to $Y$ such that $S\circ \pi=\pi\circ T$.
In this situation $(X,T)$ is said to be an \emph{extension} of $(Y,S)$
and $(Y,S)$ the \emph{factor} of $(X,T)$.

Let $\F$ be a family and $(X, T)$ be a system,
a point $x\in X$ is called an \emph{$\F$-recurrent point}
if for every open neighborhood $U$ of $x$
the entering time set $N(x,U)=\{n\in \N: T^nx\in U\}\in \F$.
If $x$ is $\F$-recurrent, then so is $Tx$. Let $\pi: (X,T)\to (Y,S)$ be a factor map,
if $x\in X$ is $\F$-recurrent, then so is $\pi(x)$.
It is well known that $x$ is recurrent if and only if it is $\F_{ip}$-recurrent and
$x$ is a minimal point if and only if it is $\F_{s}$-recurrent.
If $\F$ is a filter, then $x$ is $\F$-recurrent if and only if $\F$-$\lim T^nx=x$.

Now we generalize the notion of $\omega$-limit set.
Let $\F$ be a family, $(X, T)$ be a dynamical system and $x\in X$,
a point $y\in X$ is called an \emph{$\F$-$\omega$-limit point of $x$} if
for every neighborhood $U$ of $y$ the entering time set $N(x,U)\in \F$.
Denote by $\omega_\F(x,T)$ the set of all $\F$-$\omega$-limit points.
Then $x$ is $\F$-recurrent if and only if $x\in\omega_\F(x,T)$.

An \emph{invariant measure} for a dynamical system $(X,T)$ is a regular
Borel probability measure $\mu$ on $X$ such that
$\mu(T^{-1}A)=\mu(A)$ for all Borel subsets $A$ of $X$.

\begin{lem} \label{lem:Fps-Fpubd-rec}
Let $(X,T)$ be a dynamical system and $x\in X$. If $x$ is a recurrent point with $\overline{Orb(x,T)}=X$, then
\begin{enumerate}
\item $x$ is $\F_{ps}$-recurrent if and only if $(X,T)$ has dense minimal points (\cite{HY05}).
\item $x$ is $\F_{pubd}$-recurrent if and only if for every open neighborhood $U$ of $x$
there exists an invariant measure $\mu$ on $(X,T)$ such that $\mu(U)>0$ (\cite{HPY07,BD08}).
\end{enumerate}
\end{lem}

\begin{lem}\label{lem:F-rec-bN}
Let $\F$ be a family and $p\in\bN$.
\begin{enumerate}
\item If $p$ is an idempotent and $p\subset \F$, then $p$ is $\F$-recurrent in $(\bZ,\lambda_1)$.
\item If $p$ is $\F$-recurrent in $(\bZ,\lambda_1)$, then $p\subset b\F$.
\end{enumerate}
\end{lem}
\begin{proof}
(1) For every neighborhood $U$ of $p$, there exists some $F\in p$ such that $\overline{F}\subset U$.
Then $N(p, \overline{F})=\{n\in\N: (\lambda_1)^n p\in \overline{F}\}=\{n\in\N: n+p\in\overline{F}\}
=\{n\in\N: -n+F\in p\}$. Since $F\in p=p+p$, then $\{n\in\N: -n+F\in p\}\in p$.
Thus $N(p,\overline{F})\in \F$ and $p$ is $\F$-recurrent.

(2) For every $F\in p$, $\overline{F}$ is an open neighborhood of $p$ and
$N(0,\overline{F})=F$. Let $F'=N(p, \overline{F})$.
Since $p$ is $\F$-recurrent, $F'\in\F$. For every finite subset $W$ of $F'$,
by the continuity of $\lambda_1$, there exists an open neighborhood $U$ of $p$
such that $(\lambda_1)^nU\subset \overline{F}$ for every $n\in W$.
Since $p\in \overline{Orb(0,\lambda_1)}$, there exists some $m\in\Z$ such that $(\lambda_1)^m 0 \in U$.
Then $m+W\subset N(0,\overline{F})$. Thus, $F\in b\F$.
\end{proof}

Let $(X,T)$ be a dynamical system. Then $(X^X,T)$ also forms a dynamical system,
where $X^X$ is endowed with its compact, pointwise convergence topology
and $T$ acts on $X^X$ as composition.
The \emph{enveloping semigroup} of $(X,T)$, denoted by $E(X,T)$, is defined as the
closure of the set $\{T^n: n\in\Z\}$ in $X^X$.

From the algebraic point of view, $E(X,T)$ is a compact Hausdorff right topological semigroup.
On the other hand, $(E(X,T),T)$ is a subsystem of $(X^X,T)$.
Those two structures are closely related.
A subset $L\subset E(X,T)$ is a closed left ideal of $E(X,T)$
if and only if $(L,T)$ is a subsystem of $(E(X,T),T)$,
and $L$ is a minimal left ideal of $E(X,T)$ if and only if
$(L,T)$ is a minimal subsystem of $(E(X,T),T)$.

If $\pi: (X,T)\to (Y,S)$ is a factor map, then there is a unique continuous semigroup
homomorphism $\tilde{\pi}: E(X,T)\to E(Y,S)$ such that
$\pi(px)=\tilde{\pi}(p)\pi(x)$.

Let $(X,T)$ be a dynamical system and $I$ be any nonempty set.
Let $X^I$ be the product space and define
$T^{(I)}:X^I\to X^I$ by $T^{(I)}((x_i)_{i\in I})=(Tx_i)_{i\in I}$.
Then there is a natural isomorphism between $E(X,T)$ and $E(X^I, T^{(I)})$.
For convenience, we regard $E(X,T)$ acting on the factors of $(X,T)$
and the product systems of $(X,T)$.

For each $x\in X$, there is a canonical factor map
\[\varphi_x: E(X,T)\to (\overline{Orb(x,T)},T),\quad q\mapsto qx.\]

Let $(X, T)$ be a dynamical system. $\Z$ acts on $X$ as
\[\Phi: \Z\times X\to X,\quad (n,x)\mapsto T^n x. \]
Since $\bZ$ is the Stone-\v{C}ech compactification of $\Z$, we can extend $\Phi$ to
\[ \bZ\times X\to X,\quad (p,x)\mapsto px. \]
For each $x\in X$, the map $\Phi_x: (\bZ,\lambda_1)\to (\overline{Orb(x,T)},T)$,
$p\mapsto px$ is a factor map
and $\Phi_x(\bN^*)=\omega(x,T)$, where $\bN^*=\bN\setminus \N$.

\begin{lem}\label{lem:p-lim}
Let $(X, T)$ be a system, $x \in X$ and $p \in\bN$. Then $px = p\, \textrm{-}\lim T^nx$.
\end{lem}
\begin{proof}
Clearly, the result holds for principle ultrafilters. Now we assume that $p$ is a non-principle ultrafilter.
Consider the factor map
\[\Phi_x:\, (\bZ, \lambda_1)\to (\overline{Orb(x,T)},T), \quad p\mapsto px.\]
For every neighborhood $U$ of $px$, let $V=\Phi_x^{-1}(U)$,
then $V$ is a neighborhood $p$.
There exists a subset $F$ of $\N$ such that $p\in\overline{F}\subset V$.
Then $F\subset N(0, V)\subset N(x, U)$. Thus, $N(x,U)\in p$.
\end{proof}

We can also extend $\Psi: \Z\to X^X$, $n\mapsto T^n$ to $\bZ\to E(X,T)$.
It is easy to see that $\Psi$ is a semigroup homomorphism and
$\Psi: (\bZ,\lambda_1)\to E(X,T)$ is also a factor map.
For every $x\in X$, $\Phi_x$ and $\varphi_x\circ\Psi$ agree on $\Z$ which is dense in $\bZ$,
then $\Phi_x=\varphi_x\circ\Psi$, i.e., the following diagram commutes.
\[
\xymatrix{
(\bZ,\lambda_1) \ar[r]^-{\Psi} \ar[d]^-{\Phi_x} & (E(X,T),T) \ar[dl]^-{\varphi_x} \\
(\overline{Orb(x,T)},T) &
}
\]

Before continuing discussion, we need some preparation about
symbolic dynamics.
Let $\Sigma_2=\{0,1\}^{\Z}$ and $\sigma: \Sigma_2\to\Sigma_2$ by the shift map,
i.e. the map
\[(x(0),x(1),x(2),\ldots)\mapsto (x(1),x(2),x(3),\ldots).\]
Let $[i_0 i_1 \ldots i_n]=\{x\in \Sigma_2: x(0)=i_0,x(1)=i_1, \cdots, x(n)=i_n\}$
for $i_j\in {0,1}$ and $j=0,1,\ldots, n$.
For any $F\subset \Z$, we denote $\mathbf{1}_F$ be the indicator function from $\Z$ to $\{0,1\}$, i.e.,
$\mathbf{1}_F(n)=1$ if $n\in F$ and
$\mathbf{1}_F(n)=0$ if $n\not\in F$.
In a natural way, each indicator function can be regarded as an element of $\Sigma_2$.
It should be noticed that the enveloping semigroup of $(\{0,1\}^\Z,\sigma)$ is
topologically and algebraically isomorphic to $\bZ$ (\cite{E69,G80}).
Similarly, we can define two sided symbolic dynamics $(\{0,1\}^{\mZ},\sigma)$.

\begin{thm}\label{thm:F-rec-ide}
Let $\F$ be a filterdual. Suppose that $h(\F)$ is a subsemigroup of $\bN$.
Let $(X,T)$ be a system and $x\in X$. Then the following conditions are equivalent:
\begin{enumerate}
\item\label{enum:F-rec} $x\in X$ is an $\F$-recurrent point;
\item\label{enum:idem-KF} there exists an idempotent $u\in h(\F)$ such that $ux=x$;
\item\label{enum:idem-E} there exists an $\F$-recurrent idempotent $v\in E(X,T)$ such that $vx=x$;
\item\label{enume:ess-F-rec} $x$ is an $\widetilde{\F}$-recurrent point, where $\widetilde{\F}$ is the
collection of all essential $\F$-sets.
\end{enumerate}
\end{thm}
\begin{proof}
(\ref{enum:F-rec}) $\Rightarrow$ (\ref{enum:idem-KF})
Let
\[\mathcal A=\{N(x,U): U\textrm{ is an open neighborhood of }x\}.\]
Then $\mathcal A\subset \F$ and the intersection of any finite elements is also in $\mathcal A$.
By Lemma \ref{lem:filter-dual-ultra} there exists some $p\in h(\F)$ such that $\mathcal A\subset p$,
thus $px=x$.

Let $L=\{q\in \bN: qx=x\}$. Then $L$ is a closed subsemigroup of $\bN$ and so is $L\cap h(\F)$
since $p\in L\cap h(\F)$. By Ellis-Namakura Theorem there exists an idempotent $u\in L\cap h(\F)$.

(\ref{enum:idem-KF}) $\Rightarrow$ (\ref{enum:idem-E})
Let $v=\Psi(u)$. Since $u$ is $\F$-recurrent, $v$ is also $\F$-recurrent.
Since $\Psi$ is a semigroup homomorphism,
$vv=\Psi(u)\Psi(u)=\Psi(u u)=\Psi(u)=v$.
By $\Phi_x=\varphi_x\circ\Psi$, $x=ux=\Phi_x(u)=\varphi_x(\Psi(u))=\varphi_x(v)=vx$.

(\ref{enum:idem-KF}) $\Rightarrow$ (\ref{enume:ess-F-rec}), (\ref{enum:idem-E}) $\Rightarrow$ (\ref{enum:F-rec}) and
(\ref{enume:ess-F-rec}) $\Rightarrow$ (\ref{enum:F-rec}) are obvious.
\end{proof}

\begin{prop}\label{prop:F-rec-lift}
Let $\F$ be a filterdual. Suppose that $h(\F)$ is a subsemigroup of $\bN$.
Let $\pi: (X,T)\to (Y,S)$ is a factor map. If $y\in Y$ is an $\F$-recurrent point,
then there is an $\F$-recurrent point $x$ in $\pi^{-1}(y)$.
\end{prop}
\begin{proof}
By Theorem \ref{thm:F-rec-ide} there exists an idempotent $u\in h(\F)$ such that $uy=y$.
Choose $z\in \pi^{-1}(y)$ and let $x=uz$. Then $\pi(x)=\pi(uz)=u\pi(z)=uy=y$
and $ux=uuz=uz=x$, so $x$ is $\F$-recurrent and $x\in\pi^{-1}(y)$.
\end{proof}

\begin{rem}
Recall that a point $x\in X$ is a minimal point if and only if it is $\F_s$-recurrent.
Unfortunately, $\F_s$ is not a filterdual. Can we use some
filterdual instead of $\F_s$ to characterize minimal points?
Intuitively, $\F_{cen}$ may be a good choice. But this is not true,
it is shown in \cite{L11} that there exists an $\F_{cen}$-recurrent point which is not a minimal point.
\end{rem}

Let $(X,T)$ be a system and $x,y\in X$. We call $x,y$ are \emph{proximal} if there exists
some point $z\in X$ such that $(z,z)\in\omega((x,y),T\times T)$.

\begin{prop}[\cite{E69,F81,BH90}] \label{prop:min-id-cen}
Let $(X,T)$ be a system and $x,y\in X$. Then the following conditions are equivalent:
\begin{enumerate}
\item $x,y$ are proximal and $y$ is a minimal point;
\item there exists a minimal idempotent $u\in\bN$ such that $ux=uy=y$;
\item there exists a minimal idempotent $v\in E(X,T)$ such that $vx=vy=y$;
\item $(y,y) \in\omega_{\F_{cen}}((x,y),T\times T)$.
\end{enumerate}
\end{prop}

Let $(X,T)$ be a system and $x,y\in X$. We call $x$ is \emph{strongly proximal} to $y$
if $(y,y)\in\omega((x,y),T\times T)$.
It is easy to see that if $y$ is a minimal point then $x,y$ are proximal if and only if
$x$ is strongly proximal to $y$.

\begin{lem}\label{lem:str-porx-IP}
Let $(X,T)$ be a dynamical system and $x,y\in X$.
Then the following conditions are equivalent:
\begin{enumerate}
\item $x$ is strongly proximal to $y$;
\item $(y,y)\in \omega_{\F_{ip}}((x,y),T\times T)$;
\item for every $n\in\N$, $x$ is strongly proximal to $y$ in $(X,T^n)$.
\end{enumerate}
\end{lem}
\begin{proof}
(2)$\Rightarrow$(1) and (3)$\Rightarrow$(1) are obvious.

(2)$\Rightarrow$(3) follows from the fact that if $F$ is an IP set then
for every $n\in \N$ the set $\{m\in\N: mn\in F\}$ is also an IP set.

(1)$\Rightarrow$(2)
Consider the factor map
\[\Phi_{(x,y)}: (\bZ,\lambda_1)\to (\overline{Orb((x,y),T\times T)},T\times T),\quad q\mapsto q(x,y).\]
Let $L=\{p\in\bN: p(x,y)=(y,y)\}=\Phi_{(x,y)}^{-1}(y,y)\bigcap \bN$.
Then $L$ is not empty closed subset of $\bN$, since $(y,y)\in \omega((x,y),T\times T)$.
We show that $L$ is subsemigroup of $\bN$. Let $p,q\in L$. Then $p(x,y)=(px,py)=(y,y)$ and
$q(x,y)=(qx,qy)=(y,y)$, so $pq(x,y)=(pqx,pqy)=(py,py)=(y,y)$.
By Ellis-Namakura Theorem there exists an idempotent $p$ in $L$.
Then by Lemma \ref{lem:p-lim} and $p\subset \F_{ip}$ one has
$(y,y)\in \omega_{\F_{ip}}((x,y),T\times T)$.
\end{proof}

Let $\F$ be a family, $(X,T)$ be a system and $x,y\in X$,
we call $x$ is \emph{$\F$-strongly proximal} to $y$ if $(y,y)\in\omega_\F((x,y),T\times T)$ (\cite{A97}).

\begin{thm}\label{thm:fd-sp-fr}
Let $\F$ be a filterdual. Suppose that $h(\F)$ is a subsemigroup of $\bN$.
Let $(X,T)$ be a system and $x,y\in X$. Then the following conditions are equivalent:
\begin{enumerate}
\item\label{enum:F-str-pox} $x$ is $\F$-strongly proximal to $y$;
\item\label{enum:id-KF} there exists an idempotent $u\in h(\F)$ such that $ux=uy=y$;
\item\label{enum:id-EXT} there exists an $\F$-recurrent idempotent $v\in E(X,T)$ such that $vx=vy=y$;
\item\label{enum:neighor-ess-F} $x$ is $\widetilde{\F}$-strongly proximal to $y$.
\end{enumerate}

\end{thm}
\begin{proof}
(\ref{enum:F-str-pox}) $\Rightarrow$ (\ref{enum:id-KF})
Let \[\mathcal A=\{N((x,y), U\times U): U\textrm{ is an open neighborhood of }y\}.\]
By the definition of $\F$-strong-proximity, we have $\mathcal A\subset \F$
and the intersection of finite elements in $\mathcal A$ is also in $\mathcal A$.
Then by Lemma \ref{lem:filter-dual-ultra} there exists some $p\in h(\F)$
such that $\mathcal A\subset p$, so $p(x,y)=(y,y)$.
Let $L=\{q\in \bN: qx=qy=y\}$. Then $L\cap h(\F)$ is a nonempty closed subsemigroup of $\bN$.
By Ellis-Namakura Theorem there exists an idempotent $u\in L\cap h(\F)$.

(\ref{enum:id-KF}) $\Rightarrow$ (\ref{enum:id-EXT})
Let $v=\Psi(u)$. Since $u$ is $\F$-recurrent, $v$ is also $\F$-recurrent.
Then by $\Phi_{(x,y)}=\varphi_{(x,y)}\circ\Psi$ we have $vx=vy=y$.

(\ref{enum:id-EXT}) $\Rightarrow$ (\ref{enum:id-KF})
By Theorem \ref{thm:F-rec-ide} there exists an idempotent $u\in h(\F)$
such that $v=uv=\Psi(u)$.
Then by $\Phi_{(x,y)}=\varphi_{(x,y)}\circ\Psi$ we have $ux=uy=y$.

(\ref{enum:id-KF}) $\Rightarrow$ (\ref{enum:neighor-ess-F})
Since $u(x,y)=(y,y)$ and $u$ is an idempotent in $h(\F)$,
by Lemma \ref{lem:p-lim} $(y,y)\in\omega_{\widetilde{\F}}((x,y),T\times T)$.

(\ref{enum:neighor-ess-F}) $\Rightarrow$ (\ref{enum:F-str-pox}) is obvious.
\end{proof}

\begin{prop}\label{prop:bF-st-pro}
Let $\F$ be a filterdual. Suppose that $b\F=\F$. Let $(X,T)$ be a system and $x,y\in X$.
Then $x$ is $\F$-strongly proximal to $y$ if and only if
$y$ is an $\F$-recurrent point and $x$ is strongly proximal to $y$.
\end{prop}
\begin{proof}
By the definition, if $x$ is $\F$-strongly proximal to $y$,
then $y$ is an $\F$-recurrent point and $x$ is strongly proximal to $y$.

Conversely, assume that $y$ is an $\F$-recurrent point and $x$ is strongly proximal to $y$.
Consider the factor map
\[\Phi_{(x,y)}: (\bZ,\lambda_1)\to (\overline{Orb((x,y), T\times T)},T\times T),\ p\mapsto p(x,y).\]
Since $(y,y)\in \overline{Orb((x,y),T\times T)}$ and $(y,y)$ is $\F$-recurrent,
by Proposition \ref{prop:F-rec-lift} there exists an $\F$-recurrent point $q$ in $\bN$
with $q(x,y)=(y,y)$. By Lemma \ref{lem:F-rec-bN} we have $q\subset b\F=\F$,
then $(y,y) \in\omega_\F((x,y),T\times T)$.
\end{proof}

Now we can set up a general correspondence between essential $\F$-sets and the sets defined by $\F$-strong proximity.

\begin{thm}\label{thm:ess-F-dyna}
Let $\F$ be a filterdual. Suppose that $h(\F)$ is a subsemigroup of $\bN$.
Then a subset $F$ of $\N$ is an essential $\F$-set
if and only if there exists a dynamical system $(X,T)$, a pair of points $x,y\in X$ where $x$ is $\F$-strongly proximal to $y$,
and an open neighborhood $U$ of $y$ such that $F=N(x,U)$.
\end{thm}
\begin{proof}
The sufficiency follows from Theorem \ref{thm:fd-sp-fr} and $ N((x,y), U\times U)\subset N(x,U)$.

Now we show the necessity. If $F$ is an essential $\F$-set,
there exists an idempotent $u\in h(\F)$ such that $F\in u$.
Let $x=\mathbf{1}_F\in \{0,1\}^{\Z}$ and $y=ux$.
Then $ux=y=y$, so $x$ is $\F$-strongly proximal to $y$.
Clearly, $N(x,[1])=F$. Then it suffices to show that $y\in [1]$.
If not, then $y\in [0]$. Thus, $N(x,[0])\in p$ and $N(x,[0])\cap N(x,[1])\neq\emptyset$,
This is a contradiction.
\end{proof}

\begin{rem}
(1) In the proof of Theorem \ref{thm:ess-F-dyna}, if we use $\{0,1\}^{\mZ}$ instead of $\{0,1\}^\Z$,
then it shows that every essential $\F$-set can be realized by an invertible metrizable system.

(2) Since $\F_{ps}$ and $\F_{pubd}$ are filterduals, and $b\F_{ps}=\F_{ps}$, $b\F_{pubd}=\F_{pubd}$,
Theorem \ref{thm:quasi-cen-dy} and Theorem \ref{thm:D-set-dyn}
are special cases of Theorem \ref{thm:ess-F-dyna}.
\end{rem}

We now give a combinatorial characterization of essential $\F$-set.
\begin{prop}
Let $\F$ be a filterdual. Suppose that $h(\F)$ is a subsemigroup of $\bN$.
Then a subset $F$ of $\N$ is an essential $\F$-set if and only if
there is a decreasing sequence $\{C_n\}_{n=1}^\infty$ of subsets of $F$ such that
for every $n\in \N$, $C_n \in \F$  and for every $r\in C_n$
there exists $m\in\N$ such that $r+C_m\subset C_n$.
\end{prop}
\begin{proof}
If $F$ is an essential $\F$-set,
there exists an idempotent $u\in h(\F)$ such that $F\in u$.
Let $x=\mathbf{1}_F\in\{0,1\}^{\Z}$ and $y=ux$.
Then $u(x,y)=(y,y)$, $y\in [1]$ and $N(x,[1])=F$.
For each $n\in \N$, let $U_n=[y(0)y(1)\ldots y(n)]$ and $C_n=N((x,y), U_n\times U_n)$,
then by Theorem \ref{thm:fd-sp-fr} each $C_n$ is an essential $\F$-set.
For every $r\in C_n$, we have  $(\sigma\times \sigma)^r (y,y)\in U_n\times U_n$.
By the continuity of $\sigma$, there exists $m\in\N$ such that
$(\sigma\times \sigma)^r (U_m\times U_m)\subset U_n\times U_n$,
then $r+C_m\subset C_n$.

Conversely, assume that there is a decreasing sequence $\{C_n\}_{n=1}^\infty$ satisfying the condition.
By Lemma \ref{lem:filter-dual-ultra} there exists some $p\in h(\F)$ such that $\{C_n: n\in\N\}\subset p$.
Let $L=\bigcap_{n=1}^\infty \overline{C_n}$.
By Lemma \ref{lem:col-semigroup} $L$ is a closed subsemigroup of $\bN$.
Then $p\in L\cap h(\F)$ and $L\cap h(\F)$ is nonempty closed subsemigroup of $\bN$.
By Ellis-Namakura Theorem there exists an idempotent in $L\cap h(\F)$.
Thus, each $C_n$ is an essential $\F$-set.
In particular, $F$ is an essential $\F$-set.
\end{proof}

\begin{cor}
Let $p$ be an idempotent $\bN$ and $F\subset \N$. Then $F\in p$ if and only if
there is a decreasing sequence $\{C_n\}_{n=1}^\infty$ of subsets of $F$ such that
for every $n\in \N$, $C_n\in p$ and for every $r\in C_n$
there exists $m\in\N$ such that $r+C_m\subset C_n$.
\end{cor}

\section{The set's forcing}
In this section, we discuss the set's forcing. This terminology was first introduced in \cite{BF02},
the idea goes back at least to \cite{EK72} and \cite{G80}.
We say that a subset $F$ of $\N$ \emph{forces $\F$-recurrence} if for every dynamical system $(X,T)$ and
$x\in X$ there exists some
$\F$-recurrent point in $\overline{T^Fx}$, where $T^Fx=\{T^nx: n\in F\}$.

In  \cite{EK72} and \cite{G80}, the authors call a subset $F$ of $\N$ is \emph{big} if
there exists a minimal point in $\overline{Orb(x,\sigma)}\cap[1]$, where $x=\mathbf{1}_F \in\Sigma$.

\begin{prop}[\cite{G80,BF02}]
Let $F$ be a subset of $\N$. Then the following conditions are equivalent:
\begin{enumerate}
\item $F$ is big;
\item $F$ is piecewise syndetic;
\item $F$ forces $\F_s$-recurrence;
\item there exists a minimal left ideal $L$ of $\bN$ such that $\overline{F}\cap L\neq\emptyset$.
\end{enumerate}
\end{prop}

Let $\F$ be a family, denote by $\ff$ the collection of all sets that force $\F$-recurrence.
Clearly, $\ff$ is a family. It is easy to see that $\ff$ is not empty if and only if
there exists some $\F$-recurrent point in $(\bZ,\lambda_1)$.

\begin{thm}
Let $\F$ be a family and $F$ be a subset of $\N$.
Then $F\in \ff$ if and only if there exists an $\F$-recurrent point $p\in \bN$
such that $F\in p$.
\end{thm}
\begin{proof}
Let $F\in \ff$. Consider the system $(\bZ,\lambda_1)$ and $0\in \bZ$,
since $F$ forces $\F$-recurrence, there exists an $\F$-recurrent point
$p\in \overline{(\lambda_1)^F0}=\overline{\{(\lambda_1)^n0: n\in F\}}=\overline{F}$.
Thus, $F\in p$.

Conversely, assume that there exists an $\F$-recurrent point $p\in \bN$ such that $F\in p$.
For every dynamical system $(X,T)$ and $x\in X$,
consider the factor map $\Phi_x: (\bZ,\lambda_1)\to (\overline{Orb(x,T)},T)$.
Let $y=px$. Then $y$ is $\F$-recurrent. Then it suffices to show that $y\in \overline{T^Fx}$.
For every open neighborhood $U$ of $y$, $N(x,U)\in p$. Since $F\in p$, $N(x,U)\cap F\neq\emptyset$,
thus $y\in \overline{T^Fx}$.
\end{proof}

\begin{cor}
Let $\F$ be a family. Then
\[h(\ff)=\overline{\bigcup\{\bZ+p: p \textrm{ is an }\F\textrm{-recurrent point}\}}.\]
\end{cor}

\begin{prop}\label{prop:ff-prop}
Let $\F$ be a family. If $\ff$ is not empty,
then $\ff$ is a filterdual and $\ff=b(\ff)\subset b\F$.
\end{prop}
\begin{proof}
Let $F\in\ff$ and $F=F_1\cup F_2$. If both $F_1$ and $F_2$ are not in $\ff$,
then there exist two dynamical systems $(X,T)$ , $(Y,S)$ and two points $x\in X$, $y\in Y$
such that both $\overline{T^{F_1}x}$ and $\overline{S^{F_2}x}$ contain no $\F$-recurrent points.
Consider the system $(X\times Y, T\times S)$ and $(x,y)\in X\times Y$.
Since $F$ forces $\F$-recurrence, there exists an $\F$-recurrent point
\[(z_1,z_2)\in \overline{(T\times S)^F(x,y)}=
\overline{(T\times S)^{F_1}(x,y)}\bigcup\overline{(T\times S)^{F_2}(x,y)}.\]
Without loss of generality, assume that $(z_1,z_2)\in \overline{(T\times S)^{F_1}(x,y)}$.
Then $z_1\in \overline{T^{F_1}x}$ and $z_1$ is $\F$-recurrent, this is a contradiction.
Thus, $\ff$ is a filterdual.

Let $F\in b(\ff)$. Then there exists a sequence $\{a_n\}$ in $\Z$  and $F'\in \ff$
such that $\bigcup_{n=1}^\infty(a_n+ F'\cap [1,n])\subset F$.
Let $(X,T)$ be a dynamical system and $x\in X$.
Since $X$ is compact, there is a subnet $\{a_{n_i}\}$ of $\{a_n\}$
such that $\lim T^{a_{n_i}}x=y$.
Since $F$ forces $\F$-recurrence, there exists an $\F$-recurrent point $z\in \overline{T^{F'}y}$.
Then it suffices to show that $z\in \overline{T^F x}$.
For every open neighborhood $U$ of $z$, there exists $k\in F'$ such that $T^k y\in U$.
By the continuity of $T$, choose an open neighborhood $V$ of $y$ such that $T^k V\subset U$.
Since $\lim T^{a_{n_i}}x=y$ and $\{a_{n_i}\}$ is a subnet of $\{a_n\}$,
there exists some $n\geq k$ such that $T^{a_n} x\in V$.
Then $a_n+k\in F$ and $T^{a_n+k}x\in U$, so $z\in\overline{T^Fx}$.

Let $F\in \ff$, we show that $F\in b\F$. Let $x=\mathbf{1}_F\in \{0,1\}^\Z$.
Since $F$ forces $\F$-recurrence, there exists an $\F$-recurrent point
$y\in \overline{T^Fx}$.
Clearly, $y\in[1]$ and $N(x,[1])=F$. Let $N(y,[1])=F'$. Then $F'\in\F$.
For every finite subset $W$ of $F'$, by the continuity of $\sigma$,
there exists an open neighborhood $U$ of $y$ such that
$\sigma^n(U)\subset [1]$ for every $n\in W$.
Since $y\in \overline{Orb(x,\sigma)}$,
choose $m\in\Z$ such that $\sigma^mx\in U$, then $m+W\subset N(x,[1])$.
So $F\in b\F$.
\end{proof}

\begin{thm}\label{thm:forc-F-rec}
Let $\F$ be a filterdual and $F$ be a subset of $\N$.
Suppose that $h(\F)$ is a subsemigroup of $\bN$.
Then the following conditions are equivalent:
\begin{enumerate}
\item \label{enum:F-force-F-rec}$F$ forces $\F$-recurrence;
\item \label{enum:F-in-Sigma} let $x=\mathbf{1}_F\in\{0,1\}^{\Z}$, there exists an $\F$-recurrent point in
        $\overline{Orb(x, \sigma)}\bigcap [1]$;
\item \label{enum:F-block-ess} $F$ is a block essential $\F$-set, i.e., $F\in b\widetilde{\F}$.
\end{enumerate}
\end{thm}
\begin{proof}

(\ref{enum:F-force-F-rec}) $\Rightarrow$ (\ref{enum:F-in-Sigma})
Let $x=\mathbf{1}_F\in\{0,1\}^{\Z}$.
Since $F$ forces $\F$-recurrence,
there exists an $\F$-recurrent point $y$ in $\overline{\sigma^Fx}\subset [1]$.

(\ref{enum:F-in-Sigma}) $\Rightarrow$ (\ref{enum:F-block-ess})
Choose an $\F$-recurrent point $y$ in $\overline{Orb(x, \sigma)}\bigcap [1]$.
By Theorem \ref{thm:F-rec-ide} $y$ is also $\widetilde{\F}$-recurrent.
Since $N(x,[1])=F$ and $N(y,[1])\in \widetilde{\F}$,
by the continuity of $\sigma$ we have $F\in b\widetilde{\F}$.

(\ref{enum:F-block-ess}) $\Rightarrow$ (\ref{enum:F-force-F-rec})
By Proposition \ref{prop:ff-prop}, it suffices to show every essential $\F$-set forces $\F$-recurrence.
Let $F\in \widetilde{\F}$. Then there exists an idempotent $u\in h(\F)$ such that $F\in u$.
Let $(X,T)$ be a dynamical system and $x\in X$. Let $y=ux$. Then $uy=y$, so $y$ is $\F$-recurrent.
For every open neighborhood $U$ of $y$, $N(x,U)\in u$. Since $F\in u$, $F\cap N(x,U)\neq\emptyset$,
thus $y\in \overline{T^Fx}$.
\end{proof}

\begin{cor}
Let $\F$ be a filterdual and $F$ be a subset of $\N$.
Suppose that $h(\F)$ is a subsemigroup of $\bN$.
Let $(X,T)$ be a dynamical system and $x\in X$.
Then $x$ is the unique $\F$-recurrent point in $(X,T)$ if and only if
for every $y\in X$, $\kappa(b\widetilde{\F})$-$\lim T^n y=x$.
\end{cor}
\begin{proof}
Since $\F$ is a filterdual, $\kappa(b\widetilde{\F})$ is a filter.
If $x$ is the unique $\F$-recurrent point,
then by Theorem \ref{thm:forc-F-rec} for every $y\in X$ and every $F\in b\widetilde{\F}$
we have $x\in \overline{T^Fy}$, so $\kappa(b\widetilde{\F})$-$\lim T^n y=x$.

Conversely, assume that there exists another $\F$-recurrent point $y\in X$.
Choose open subsets $U$, $V$ of $X$ such that $x\in U$, $y\in V$ and $U\cap V=\emptyset$.
Then $N(y,U)\in \kappa(b\widetilde{\F})$ and $N(y, V)\in \widetilde{\F}\subset b\widetilde{\F}$.
Thus, $N(y,U)\cap N(y,V)\neq\emptyset$. This is a contradiction.
\end{proof}

\begin{rem}
(1) Since $\F_{ip}=\widetilde{\F_{ip}}$, $b\F_{ip}=b\widetilde{\F_{ip}}$.
Then a subset $F$ of $\N$ forces recurrence if and only if $F\in b\F_{ip}$ (\cite{BF02}).

(2) It is shown in \cite{S04} that a subset $F$ of $\N$ forces $\F_{pubd}$-recurrence if and only if $F\in \F_{pubd}$,
i.e., $b\widetilde{\F_{pubd}}=\F_{pubd}$. For completeness, we include a proof.
Let $F\in \F_{pubd}$ and $x=\mathbf{1}_F\in\{0,1\}^\Z$.
By \cite[Lemma 3.17]{F81}, there exists a $\sigma$-invariant measure $\mu$ such that
$\mu(\overline{Orb(x,\sigma)}\bigcap [1])>0$. By the ergodic decomposition theorem, choose
an ergodic $\sigma$-invariant measure $\nu$ such that $\nu(\overline{Orb(x,\sigma)}\bigcap [1])>0$.
Then pick a generic point $y$ in $\overline{Orb(x,\sigma)}\bigcap [1]$ for $\nu$ and so
$y$ is $\F_{pubd}$-recurrent (\cite[pp. 62-64]{F81}).
Thus, $F$ forces $\F_{pubd}$-recurrent.
\end{rem}

It is interesting that central sets also have some kind of forcing.

\begin{prop}
Let $F$ be a subset of $\N$. Then the following conditions are equivalent:
\begin{enumerate}
\item $F$ is central;
\item let $x=\mathbf{1}_F\in\{0,1\}^\Z$, there exists some minimal point $y\in\overline{Orb(x,\sigma)}\cap[1]$
such that $x$, $y$ are proximal;
\item for every dynamical system $(X,T)$ and $x\in X$ there exists some minimal point
$y\in \overline{T^Fx}$ such that $x$, $y$ are proximal.
\end{enumerate}
\end{prop}
\begin{proof}
(2) $\Rightarrow$ (1) follows from the definition of central sets and $N(x,[1])=F$.

(3) $\Rightarrow$ (2) follows from $\overline{T^Fx}\subset [1]$.

(1) $\Rightarrow$ (3) If $F$ is central, then there exists a minimal idempotent $u\in\bN$ such that $F\in u$.
Let $(X,T)$ be a dynamical system and $x\in X$. Let $y=ux$. Then $ux=uy=y$,
so $y$ is a minimal point and $x$, $y$ are proximal. Thus it suffices to show that $y\in \overline{T^Fx}$.
For every open neighborhood $U$ of $y$, $N(x,U)\in u$. Since $F\in u$, $F\cap N(x,U)\neq\emptyset$,
so $y\in \overline{T^Fx}$.
\end{proof}

We say  a subset $F$ of $\N$ \emph{forces $\F$-strong proximity}
if for every dynamical system $(X,T)$ and $x\in X$ there exists some point $y$
in $\overline{T^Fx}$ such that $x$ is $\F$-strongly proximal to $y$.

\begin{prop} \label{prop:F-force-sp}
Let $\F$ be a filterdual. Suppose that $h(\F)$ be a subsemigroup of $\bN$.
Let $F$ be a subset of $\N$. Then the following conditions are equivalent:
\begin{enumerate}
\item $F$ is an essential $\F$-set;
\item let $x=\mathbf{1}_F\in\{0,1\}^\Z$, there exists some point $y\in \overline{Orb(x,\sigma)}\cap[1]$
such that $x$ is $\F$-strongly proximal to $y$;
\item $F$ forces $\F$-strong proximity.
\end{enumerate}
\end{prop}
\begin{proof}
(2) $\Rightarrow$ (1) follows from Theorem \ref{thm:ess-F-dyna} and $N(x,[1])=F$.

(3) $\Rightarrow$ (2) follows from $\overline{T^Fx}\subset [1]$.

(1) $\Rightarrow$ (3) If $F$ is an essential $\F$-set,
then there exists an idempotent $u\in h(\F)$ such that $F\in u$.
Let $(X,T)$ be a dynamical system and $x\in X$. Let $y=ux$. Then $ux=uy=y$
and by Theorem \ref{thm:fd-sp-fr} $x$ is $\F$-strongly proximal to $y$.
Thus it suffices to show that $y\in \overline{T^Fx}$.
For every open neighborhood $U$ of $y$, $N(x,U)\in u$.
Since $F\in u$, $F\cap N(x,U)\neq\emptyset$, so $y\in \overline{T^Fx}$.
\end{proof}


\section{Multiplication in $\N$ and $\bN$}
In this section, we consider both addition and multiplication in $\N$ and $\bN$.
For $n\in\N$ and $F\subset \N$, let $nF=\{nm: m\in F\}$ and $n^{-1}F=\{m\in\N: nm\in F\}$.
For $p,q\in\bN$, the multiplication $p\cdot q$ in $\bN$ is
\[\{A\subset\N: \{n\in\N: n^{-1}A\in q\}\in p\}.\]

A family $\F$ is called \emph{multiplication invariant} if for each $n\in\N$ and $F\in\F$ one has $nF\in\F$.
It is easy to see that $\F_{ip}$, $\F_{s}$ and $\F_{pubd}$ are multiplication invariant.
Similarly to Lemma \ref{lem:filterdual-tran-invar}, we have

\begin{lem}
Let $\F$ be a filterdual. Then $\F$ is multiplication invariant if and only if
$h(\F)$ is a left ideal of $(\bN,\cdot)$.
\end{lem}

\begin{prop}[\cite{F81,BHK96}] \label{prop:cen-set-n}
Let $F$ be a subset of $\N$. If $F$ is a central set,
then for each $n\in\N$ both $nF$ and $n^{-1}F$ are also central sets.
\end{prop}

The main purpose of this section is to extend Proposition \ref{prop:cen-set-n} to
more general settings. In particular, similar results hold for quasi-center sets and D-sets.

\begin{thm}\label{thm:nF-ess}
Let $\F$ be a filterdual and $F$ be a subset of $\N$.
Suppose that $\F$ is multiplication invariant and $h(\F)$ is a subsemigroup of $(\bN,+)$.
If $F$ is an essential $\F$-set, then for each $n\in\N$, $nF$ is also an essential $\F$-set.
\end{thm}
\begin{proof}
Let $x=\mathbf{1}_F\in \{0,1\}^\Z$. Then by Proposition \ref{prop:F-force-sp}  there exists a point
$y\in \overline{\sigma^Fx}\subset [1]$ such that $\F$-$\lim (\sigma\times \sigma)^m(x,y)=(y,y)$.
Fix $n\in\N$ and let $Y=\{1,2,\ldots,n\}$ endowing with discrete topology and  $X=\{0,1\}^\Z \times Y$.
Define $T: X\to X$ by $T(z,i)=(z,i+1)$ for $i\leq n-1$ and $T(z,n)=(\sigma z,1)$.

For every neighborhood $U$ of $y$, we have
\[N((x,1,y,1), U\times\{1\}\times U\times \{1\})= n N((x,y),U\times U).\]
Since $\F$ is multiplication invariant, $\F$-$\lim (T\times T)^m(x,1,y,1)=(y,1,y,1)$.
Thus, $nF=N((x,1), [1]\times \{1\})$ is also an essential $\F$-set.
\end{proof}

We call $\F$-recurrence is  \emph{iteratively invariant} if
for every dynamical system $(X,T)$ and every $\F$-recurrent point $x$ in $(X,T)$,
$x$ is also an $\F$-recurrent point in $(X,T^n)$ for each $n\in \N$.
It is well known that $\F_{ip}$-recurrence and $\F_{s}$-recurrence are iteratively invariant.
We show that

\begin{thm}\label{thm:n-1F-ess}
Let $\F$ be a filterdual and $F$ be a subset of $\N$.
Suppose that $b\F=\F$ and $\F$-recurrence is iteratively invariant.
If $F$ is an essential $\F$-set, then for each $n\in\N$, $n^{-1}F$ is also an essential $\F$-set.
\end{thm}
\begin{proof}
Let $x=\mathbf{1}_F\in \{0,1\}^\Z$. Then by Proposition \ref{prop:F-force-sp}
there exists an $\F$-recurrent point
$y\in \overline{\sigma^Fx}\subset [1]$ such that $x$ is strongly proximal to $y$.
For each $n\in \N$, since $\F$-recurrence is iteratively invariant,
$y$ is also an $\F$-recurrent point in $(\{0,1\}^\Z, \sigma^n)$.
By Lemma \ref{lem:str-porx-IP}, $x$ also is strongly proximal to $y$ in $(\{0,1\}^\Z, \sigma^n)$.
Then by Proposition \ref{prop:bF-st-pro} and Theorem \ref{thm:fd-sp-fr}
$n^{-1}F=\{m\in\N: (\sigma^n)^m x\in [1]\}$ is an essential $\F$-set.
\end{proof}

A system $(X,T)$ is called \emph{topologically transitive}
if for every two nonempty open subset $U$, $V$ of $X$ there exists some $n\in\N$ such that
$T^n U\cap V\neq\emptyset$. A point $x\in X$ is called a \emph{transitive point}
if the orbit of $x$ is dense in $X$.
The system $(X,T)$ is called \emph{point transitive} if there exists some transitive point in $X$.
In general, there is no implicational relation between topological transitivity and
point transitivity. For example, $(\bZ,\lambda_1)$ is point transitive but not topologically transitive.
The system $(X,T)$ is called \emph{recurrent transitive}
if there exists some recurrent transitive point,
i.e., there exists $x\in X$ such that $\omega$-limit set of $x$ is $X$.
It is easy to see that every recurrent point transitive system is topologically transitive.

The following result is a ``folk'' result, for similar results, see \cite{B97} for example.

\begin{lem}\label{lem:trans-decom}
Let $(X,T)$ be a recurrent transitive system. Then for every $n\in\N$
there is $k\in\N$ with $k|n$ and a decomposition $X=X_0\cup X_1\cup\cdots\cup X_{k-1}$
satisfying
\begin{enumerate}
\item $X_i\neq X_j$, $0\leq i<j\leq k-1$,
\item $TX_i=X_{i+1\pmod k}$,
\item $(X_i, T^n)$ is recurrent transitive, $i=0,\ldots,k-1$,
\item the interior of $X_i$ is dense in $X_i$, $i=0,\ldots,k-1$.
\end{enumerate}
\end{lem}
\begin{proof}
Let $x\in X$ with $\omega(x,T)=X$. Let $Y_i=\overline{Orb(T^ix, T^n)}$ for $i=0,1,\ldots, n-1$.
Then $X=Y_0\cup Y_1\cup \cdots \cup Y_{n-1}$ and $T Y_i=Y_{i+1 \pmod n}$.
Since $x$ is recurrent in $(X,T)$, $T^ix$ is also recurrent in $(X,T^n)$. Then $(Y_i,T^n)$
is recurrent transitive for $i=0,1,\ldots, n-1$.
Let $k$ be the smallest positive integer such that $T^k Y_0=Y_0$.
Let $X_i=Y_i$ for $i=0,1,\ldots,k-1$. Now we show that those $X_i$ satisfy the requirement.

Clearly, $k\leq n$. Let $n=lk+r$ with $l>0$ and $0\leq r <k$.
Then $X_0=T^n(X_0)=T^r(T^{lk}X_0)=T^r(X_0)$, by the minimality of $k$, we have $r=0$, so $k|n$.

If there exist $0\leq i<j\leq k-1$ such that $X_i=X_j$, then
$T^{j-i}X_0=T^{j-i}(T^nX_0)=T^{n-i}(T^jX_0)=T^{n-i}(T^iX_0)=T^nX_0=X_0$.
This contradicts  the minimality of $k$. So $X_i\neq X_j$ for $0\leq i<j\leq k-1$.

For $0\leq i\neq j\leq k-1$  let $Z_{ij}=X_i\cap X_j$, then $Z_{ij}$ is a $T^n$-invariant closed subset of $X_i$.
Since $(X_i,T^n)$ is topologically transitive, $Z_{ij}$ either equals to $X_i$ or is nowhere dense in $X_i$.
If $Z_{ij}=X_i$, then $X_i\subset X_j$.
Without loss of  generality, assume $i< j$, then $X_0=T^{k-i}X_i\subset T^{k-i}X_j=X_{j-i}$.
Thus,
\[X_0\subset X_{j-i}\subset X_{2(j-i) \pmod k}\subset \cdots\subset X_{k(j-i) \pmod k} =X_0.\]
This contradicts the minimality of $k$. So $Z_{ij}$ is nowhere dense in $X_i$.

Now fix $i\in\{0,1,\ldots,k-1\}$ and let $Z_i=\bigcup_{j\neq i}Z_{ij}$. Then $Z_i$ is also nowhere dense in $X_i$.
The boundary of $X_i$ in $X$ is
\[\partial X_i= X_i\cap(\overline{X\setminus X_i})\subset X_i\bigcap (\bigcup_{j\neq i}X_j)=
\bigcup_{j\neq i}(X_i\bigcap X_j)=Z_i.\]
By $X_i=int(X_i)\bigcup Z_i$, we have the interior of $X_i$ is dense in $X_i$.
\end{proof}

\begin{lem} \label{lem:Fps-Fpubd-rec-inv}
$\F_{ps}$-recurrence and $\F_{pubd}$-recurrence are iteratively invariant.
\end{lem}

\begin{proof}
Let $(X,T)$ be a dynamical system and $x\in X$ be an $\F_{ps}$-recurrent point.
Without loss of  generality, assume that $\overline{Orb(x,T)}=X$.
By Lemma \ref{lem:Fps-Fpubd-rec}, $(X,T)$ has dense minimal points.
For every $n\in \N$, $(X,T^n)$ also has dense minimal points.
By Lemma \ref{lem:trans-decom}, the interior of $\overline{Orb(x,T^n)}$ is dense in $\overline{Orb(x,T^n)}$,
so $(\overline{Orb(x,T^n)}, T^n)$ also has dense minimal points.
Thus $x$ is $\F_{ps}$-recurrent in $(X,T^n)$.

Let $(X,T)$ be a dynamical system and $x\in X$ be an $\F_{pubd}$-recurrent point.
Without loss of  generality, assume that $\overline{Orb(x,T)}=X$.
By Lemma \ref{lem:Fps-Fpubd-rec} and $(X,T)$ is transitive, for every nonempty open subset $U$ of $X$
there exists a $T$-invariant measure $\mu$ on $X$ such that $\mu(U)>0$.
For every $n\in \N$, by Lemma \ref{lem:trans-decom},
the interior of $\overline{Orb(x,T^n)}$ is dense in $\overline{Orb(x,T^n)}$.
Then for every nonempty open subset $V$ of $\overline{Orb(x,T^n)}$
there exists an open subset $U$ of $X$ such that $U\subset V$.
So there exists a $T$-invariant measure $\mu$ on $X$ such that $\mu(U)>0$.
Clearly, $\mu$ is also a $T^n$-invariant measure on $X$.
Define a measure $\nu$ on $\overline{Orb(x,T^n)}$  by $\nu(A)=\mu(A)/\mu(\overline{Orb(x,T^n)})$
for every Borel subset $A$ of $\overline{Orb(x,T^n)}$.
Then $\nu$ is a $T^n$-invariant measure on $\overline{Orb(x,T^n)}$ with $\nu(V)>0$.
Thus $x$ is $\F_{pubd}$-recurrent in  $(X,T^n)$.
\end{proof}

\begin{prop}
Let $F$ be a subset of $\N$ and $n\in\N$.
\begin{enumerate}
\item If $F$ is a quasi-central set, then both $nF$ and $n^{-1}F$ are also quasi-central sets.
\item If $F$ is a D-set, then both $nF$ and $n^{-1}F$ are also D-sets
\end{enumerate}
\end{prop}
\begin{proof}
It follows from Theorem \ref{thm:n-1F-ess}, Theorem \ref{thm:nF-ess}, Lemma \ref{lem:Fps-Fpubd-rec-inv}
and the fact that $\F_{ps}$ and $\F_{pubd}$ are multiplication invariant.
\end{proof}

\section{Dynamical characterization of C-sets}
In this section, we show the following dynamical characterization of C-sets.

\begin{thm}
Let $F$ be a subset of $\N$. Then $F$ is a C-set if and only if
there exists a dynamical system $(X,T)$, a pair of points $x,y\in X$ where $y$ is $\mathcal J$-recurrent and $x$ is strongly proximal to $y$,
and an open neighborhood $U$ of $y$ such that $N(x,U)=F$.
\end{thm}

By Proposition \ref{prop:bF-st-pro} and Theorem \ref{thm:ess-F-dyna}, it suffices to show the following two lemmas.

\begin{lem}\label{lem:J-f-d}
$\mathcal J$ is a filterdual.
\end{lem}

\begin{lem}\label{lem:bJ-J}
$\mathcal J=b\mathcal J$ and it is multiplication invariant.
Then $h(\mathcal J)$ is a closed two sided ideal in $(\bN,+)$ and left ideal in $(\bN,\cdot)$.
\end{lem}

\begin{proof}[Proof of Lemma \ref{lem:J-f-d}]
Let $F$ be a J-set and $F=F_1\cup F_2$.
Using an argument due to \cite[Theorem 2.14]{HS10}, we first show the following claim.

{\bf Claim}: For every IP-system $\{s_\alpha=(s^{(1)}_{\alpha},\ldots,s^{(m)}_\alpha)\}$ in $\mZ^m$,
there exists an $i\in\{1,2\}$, $r\in \mZ$ and $\alpha\in \pn$
such that $\bar r^{(m)}+s_{\alpha}\in F_i^m$.

{\bf Proof of the Claim}: For $j=1,2,\ldots,m$, define $f_j:\N\to\mZ$ by $f_j(n)=s^{(j)}_{\{n\}}$,
then $s^{(j)}_\alpha=\sum_{n\in\alpha}f_j(n)$ for $\alpha\in\pn$.

Pick by Hales-Jewett Theorem (\cite{HJ63}) some $n\in\N$ such that
whenever the length $n$ words over the alphabet $\{1,\ldots,m\}$ are $2$-colored,
there exists a variable word $w(v)$ such that $\{w(j): j=1,\ldots,m\}$ is monochromatic.

Let $W$ be the set of length $n$ words over $\{1,\ldots,m\}$.
For $w=b_1b_2\cdots b_n\in W$
define $g_w:\N\to\mZ$ by for $l\in\Z$ and $i=1,2,\ldots,n$, $g_w(ln+i)=f_{b_i}(ln+i)$.
For $l\in\Z$, let $H_l=\{ln+1,ln+2,\ldots,ln+n\}$.
For every $w\in W$ and $\alpha\in\pn$, let
$h^{(w)}_\alpha=\sum_{l\in\alpha}\sum_{t\in H_l}g_w(t)$.
Then $(h_\alpha)=(h^{(w)}_\alpha: w\in W)$ is an IP-system in $\mZ^{|W|}$.
Then there exists a $r\in\mZ$ and $\alpha\in \pn$
such that $r+h^{(w)}_\alpha\in F$ for every $w\in W$.
Define $\phi: W\to\{0,1\}$ by $\phi(w)=1$ if $r+h^{(w)}_\alpha\in F_1$.
Pick a variable word $w(v)$ such that $\{w(j): j=1,2,\ldots,m\}$ is monochromatic with respect to $\phi$.
Without loss of  generality assume that $\phi(w(j))=1$ for $j=1,2,\ldots,k$.
Let $w(v)=c_1c_2\cdots c_n$ where each $c_i\in\{1,2,\ldots,m\}\cup\{v\}$.
Let $A=\{i\in\{1,2,\ldots,n\}: c_i=v\}\neq\emptyset$ and $B=\{1,2,\ldots,n\}\setminus A$.
For $l\in\Z$, let $H_l^A=H_l\cap(ln+A)$ and $H_l^B=H_l\cap(ln+B)$.
For $j=1,2,\ldots,m$, rewrite $h^{(w(j))}_\alpha$ as
\[h^{(w(j))}_\alpha= \sum_{l\in \alpha}\sum_{t\in H_l}g_{w(j)}(t)
=\sum_{l\in\alpha}\sum_{t\in H_l^A}g_{w(j)}(t)
+\sum_{l\in\alpha}\sum_{t\in H_l^B}g_{w(j)}(t).\]
Then $\sum_{t\in H_l^A}g_{w(j)}(t)=\sum_{t\in H_l^A}f_j(t)$ and
$\sum_{t\in H_l^B}g_{w(j)}(t)$ does not depend on $j$.
Let $\alpha'=\bigcup_{l\in\alpha}H_l^A$
and $r'=r+\sum_{l\in\alpha}\sum_{t\in H_l^B}g_{w(j)}(t)$.
Then $r+h^{(w(j))}_\alpha =r'+s^{(j)}_{\alpha'}$.
So $\bar r'^{(m)}+s_{\alpha'}\in F_1^m$. This ends the proof of the Claim.

We now show that in the Claim we can pick $r\in\N$ instead of $r\in\mZ$.
For every IP-system $\{s_\alpha=(s^{(1)}_{\alpha},\ldots,s^{(m)}_\alpha)\}$ in $\mZ^m$,
let $s^{(0)}_\alpha= - |\alpha|$ for each $\alpha\in\pn$ and
$\{s'_\alpha= (s^{(0)}_\alpha, s^{(1)}_{\alpha},\ldots,s^{(m)}_\alpha)\}$.
Applying the Claim to $\{s'_\alpha\}$,
there exists an $i\in\{1,2\}$, $r\in \mZ$ and $\alpha\in \pn$
such that $\bar r^{(m+1)}+s'_{\alpha}\in F_i^{m+1}$.
Since $r+s^{(0)}_\alpha\in F_i$ and $s^{(0)}_\alpha$ is negative, $r$ must be positive.

If both $F_1$ and $F_2$ are not J-sets, let
$\{s_\alpha=(s^{(1)}_{\alpha},\ldots,s^{(m)}_\alpha)\}$
and $\{s'_\alpha=(s'^{(1)}_{\alpha},\ldots,s'^{(m')}_\alpha)\}$
be witnesses to the fact that $F_1$ and $F_2$ are not J-sets.
Let $s''_\alpha=(s^{(1)}_{\alpha},\ldots,s^{(m)}_\alpha, s'^{(1)}_{\alpha},\ldots,s'^{(m')}_\alpha)$.
Applying the Claim to $\{s''_\alpha\}$, we get a contradiction.
\end{proof}

\begin{proof}[Proof of Lemma \ref{lem:bJ-J}]
If $F$ is a block J-set, then there exists a sequence $\{a_n\}$ in $\Z$ and $F'\in \mathcal J$
such that $\bigcup_{n=1}^\infty(a_n+F'\cap[1,n])\subset F$.
For every IP-system $\{s_\alpha\}$ in $\mZ^m$, there exists $r\in\N$ and $\alpha\in\pn$
such that $\bar r^{(m)}+s_\alpha\in F'^{(m)}$.
Choose $n$ large enough such that $\bar r^{(m)}+s_\alpha\in (F'\cap[1,n])^{(m)}$ and
let $r'=r+a_n$. Then $\bar r'^{(m)}+s_\alpha\in F^{m}$. Hence, $F$ is also a J-set.

Let $F$ be a J-set and $n\in\N$, we want to show that $nF$ is also a J-set.
Let $\{s_\alpha\}$ be an IP-system in $\mZ^m$.
Without loss of  generality, assume that $\{s_\alpha\}\subset n\mZ^m$.
Let $s'_{\alpha}=n^{-1}s_\alpha$. Then $\{s'_\alpha\}$ is also an IP-system in $\mZ^m$.
Since $F$ is a J-set, there exists $r\in\N$ and $\alpha\in \pn$ such that
$\bar r^{(m)}+s'_{\alpha}\in F^{(m)}$, then $\bar{nr}^{(m)}+s_\alpha\in nF^{(m)}$.
Hence, $nF$ is also a J-set.
\end{proof}

\begin{rem}
It is shown in \cite{H09} that there exists a C-set with upper Banach density $0$.
Then there exists a dynamical system $(X,T)$ and $x\in X$ such that
$x$ is $\mathcal J$-recurrent but not $\F_{pubd}$-recurrent.
\end{rem}

\section{Solvability of Rado systems in C-sets}

In order to show that Rado systems are solvable in C-sets, by the method developed in \cite[pp.\@169--174]{F81},
it suffices to show the following two results.

\begin{lem}\label{lem:n-1c-set}
If $F$ is a C-set, then for each $n\in\N$, $nF$ and $n^{-1}F$ are also C-sets.
\end{lem}

\begin{thm}\label{thm:C-set-thm}
Let $F$ be a C-set. Then for every $m\in\N$ and every IP-system $\{s_\alpha\}$ in $\mZ^m$
there exists an IP-system $\{r_\alpha\}$ in $\N$ and an IP-subsystem $\{s_{\phi(\alpha)}\}$
such that for every $\alpha\in\pn$, $\bar r_\alpha^{(m)}+s_{\phi(\alpha)} \in F^m$.
\end{thm}

To discuss $\mathcal J$-recurrence, we first introduce a new kind of dynamical system.
Let $(X,T)$ be an invertible dynamical system, we say that $(X,T)$ satisfies the
\emph{multiple IP-recurrence property} if for every
IP-system $\{s_\alpha=(s^{(1)}_\alpha,\ldots,s^{(m)}_\alpha)\}$ in $\mZ^m$
and every open subset $U$ of $X$, there exists some $\alpha\in\pn$ such that
\[\bigcap_{i=1}^m T^{-s^{(i)}_\alpha} U\neq\emptyset.\]

Let $(X,T)$ be an invertible dynamical system, if it is a minimal system or
there exists an invariant measure with full support, then it
satisfies the multiple IP-recurrence property (\cite{F81,FK85}).

\begin{lem}\label{lem:mul-ip-sys-eq}
Let $(X,T)$ be an invertible dynamical system and $n\in\N$.
Then the following conditions are equivalent:
\begin{enumerate}
\item $(X,T)$ satisfies the multiple IP-recurrence property;
\item for every IP-system $\{s_\alpha=(s^{(1)}_\alpha,\ldots,s^{(m)}_\alpha)\}$ in $\mZ^m$,
every open subset $U$ of $X$ and $k\in \N$, there exists some $\alpha\in\pn$  with $\min \alpha >k$ such that
\[U\bigcap \left(\bigcap_{i=1}^m T^{-s^{(i)}_\alpha} U \right)\neq\emptyset;\]
\item $(X,T^n)$ satisfies the multiple IP-recurrence property.
\end{enumerate}
\end{lem}
\begin{proof}(1)$\Rightarrow$(3) and (2)$\Rightarrow$ (1) are obvious.

(1)$\Rightarrow$(2) Let $\{s_\alpha=(s^{(1)}_\alpha,\ldots,s^{(m)}_\alpha)\}$ be an IP-system $\mZ^m$ and $k\in\N$.
Define a homomorphism $\phi: \pn\to \pn$ determined by $\phi(\{i\})=\{i+k\}$ for any $i\in\N$.
Let $s^{(0)}_\alpha=0$ for any $\alpha\in\pn$.
Then $\{s'_{\alpha}=(s^{(0)}_\alpha, s^{(1)}_{\phi(\alpha)},\ldots,s^{(m)}_{\phi(\alpha)})\}$ is an IP-system
in $\mZ^{m+1}$. Now (2) follows from applying to $\{s'_{\alpha}\}$.

(3)$\Rightarrow$(1) Let $\{s_\alpha\}$ be an IP-system in $\mZ^m$.
Without loss of  generality, assume that $\{s_\alpha\}\subset n\mZ^m$.
Let $s'_{\alpha}=n^{-1}s_\alpha$. Then $\{s'_\alpha\}$ is also an IP-system in $\mZ^m$.
Then (1) follows from applying to $\{s'_{\alpha}\}$ in $(X,T^n)$.
\end{proof}

Let $\{x_\alpha\}_{\alpha\in\pn}$ be a sequence in topological space $X$ and $x\in X$,
we say that $x_\alpha\to x$ as a $\pn$-sequence if for every neighborhood $U$ of $x$
there exists $\alpha_U\in\pn$ such that $x_\alpha\in U$ for all $\alpha>\alpha_U$.
If $\{x_\alpha\}$ is a $\pn$-sequence in a compact metric space,
then there exists a $\pn$-subsequence $\{x_{\phi(\alpha)}\}$ which converges as
$\pn$-sequence (\cite[Theorem 8.14]{F81}).

\begin{prop}
Let $(X,T)$ be an invertible metrizable dynamical system.
Then $(X,T)$ satisfies the multiple IP-recurrent property if and only if
for every IP-system $\{s_\alpha\}$ in $\mZ^m$ and every open subset $U$ of $X$
there exists $x\in U$ and an IP-subsystem $\{s_{\phi(\alpha)}\}$ such that
$T^{s^{(i)}_{\phi(\alpha)}}x\to x$ for $i=1,\ldots,m$.
\end{prop}
\begin{proof}
The sufficiency is obvious.

We now show the necessity. Let $\{s_\alpha=(s^{(1)}_\alpha, s^{(2)}_\alpha, \ldots, s^{(m)}_\alpha)\}$
be an IP-system in $\mZ^m$ and $U$ be an open subset of $X$.
Let $U_0=U$.
By Lemma \ref{lem:mul-ip-sys-eq}, there exists $\alpha_1\in\pn$ such that
\[U_0\bigcap \left(\bigcap_{i=1}^m T^{-s^{(i)}_{\alpha_1}} U_0 \right)\neq\emptyset.\]
Then choose an open subset $U_1$ with $\overline{U_1}\subset U_0$ and $\mathrm{diam}(U_1)<1$
such that
\[\bigcup_{i=1}^m T^{s^{(i)}_{\alpha_1}} U_1\subset U_0.\]
We now proceed inductively to define a sequence of open subset $U_1$, $U_2$,
$\ldots, U_n,\ldots $ in $X$ and a sequence $\alpha_1<\alpha_2<\cdots<\alpha_n<\cdots$ in $\pn$
such that
\[\overline{U_{n+1}}\subset U_n,\quad \mathrm{diam}(U_n)<\frac{1}{n}, \quad
\bigcup_{i=1}^m T^{s^{(i)}_{\alpha_n}} U_n\subset U_{n-1}.\]
Then there will be a unique point $x$ in $\bigcap_{n=1}^\infty\overline{U_n}$.
Now set $\phi(\{n\})=\alpha_n$ for each $n\in\N$. For every $\beta=\{r_1<r_2<\cdots<r_k\}$ if $\min \beta>n+1$ then
\[\bigcup_{i=1}^m  T^{s^{(i)}_{\phi(\beta)}} U_{r_k}= \bigcup_{i=1}^m T^{s^{(i)}_{\alpha_{r_1}}}
T^{s^{(i)}_{\alpha_{r_2}}}\cdots T^{s^{(i)}_{\alpha_{r_k}}}U_{r_k}\subset U_{r_1-1}\subset U_n.\]
Hence, for $i=1,2,\ldots, m$, $T^{s^{(i)}_{\phi(\beta)}} x\in U_n$ if $\min \beta>n+1$.
It follows that $T^{s^{(i)}_{\phi(\alpha)}}x\to x$ for $i=1,2,\ldots,m$.
\end{proof}

\begin{thm}
Let $(X,T)$ be an invertible dynamical system and $x\in X$.
Then $x$ is $\mathcal J$-recurrent if and only if $(\overline{Orb(x,T)},T)$ satisfies
the  multiple IP-recurrence property.
\end{thm}
\begin{proof}
Without loss of  generality, assume that $\overline{Orb(x,T)}=X$.
If $x$ is $\mathcal J$-recurrent, then for every open subset $U$ of $X$
there exists $k\in\N$ and an open neighborhood $V$ of $x$ such that $T^k V\subset U$.
Since $x$ is $\mathcal J$-recurrent, $N(x,V)$ is a J-set.
Then for every IP-system $\{s_\alpha=(s^{(1)}_\alpha,\ldots,s^{(m)}_\alpha)\}$ in $\mZ^m$
there exists $r\in\N$ and $\alpha\in\pn$ such that $T^{r+s^{(i)}_\alpha} x\in V$ for $i=1,\ldots,m$.
Let $y=T^{r+k}x$. Then $T^{s^{(i)}_\alpha}y=T^{k}(T^{r+s^{(i)}_\alpha} x)\in T^k V\subset U$
for $i=1,\ldots, m$. So $y\in \bigcap_{i=1}^m T^{-s^{(i)}_\alpha} U$.

Conversely, assume that $(X,T)$ satisfies the  multiple IP-recurrence property.
It is easy to see that $x$ is recurrent. For every open neighborhood $U$ of $x$ and
every IP-system $\{s_\alpha=(s^{(1)}_\alpha,\ldots,s^{(m)}_\alpha)\}$ in $\mZ^m$,
there exists some $\alpha\in\pn$ such that $\bigcap_{i=1}^m T^{-s^{(i)}_\alpha} U\neq\emptyset$.
Choose $y\in \bigcap_{i=1}^m T^{-s^{(i)}_\alpha} U$,
then $T^{s^{(i)}_\alpha}y\in U$ for $i=1,\ldots,m$.
By the continuity of $T$, choose an open neighborhood $V$ of $y$ such that
$T^{s^{(i)}_\alpha}V\in U$ for $i=1,\ldots,m$.
Since $y\in\omega(x,T)$, there exists $r\in\N$ such that $T^rx\in U$ and $\bar r^{(m)}+s_\alpha\in \N^m$.
Then $\bar r^{(m)}+s_\alpha \in N(x,U)^m$. Therefore, $N(x,U)$ is a J-set.
\end{proof}

\begin{prop}\label{prop:J-rec-inv}
Let $(X,T)$ be an invertible dynamical system, $x\in X$ and $n\in\N$.
Then $x$ is $\mathcal J$-recurrent in $(X,T)$ if and only if
it is $\mathcal J$-recurrent in $(X,T^n)$.
\end{prop}
\begin{proof}
Without loss of generality, assume that $\overline{Orb(x,T)}=X$.
Since $\mathcal J$ is multiplication invariant, if $x$ is $\mathcal J$-recurrent in $(X,T^n)$,
then so is in $(X,T)$.

Conversely, if $x$ is $\mathcal J$-recurrent in $(X,T)$,
then $(X,T)$ satisfies the  multiple IP-recurrence property,  so does $(X,T^n)$.
Since the interior of $\overline{Orb(x,T^n)}$ is dense in $\overline{Orb(x,T^n)}$,
it is easy to see that $(\overline{Orb(x,T^n)},T^n)$ also satisfies the  multiple IP-recurrence property.
Then $x$ is $\mathcal J$-recurrent in $(X,T^n)$.
\end{proof}


\begin{proof}[Proof of Lemma \ref{lem:n-1c-set}]
It follows from Theorem \ref{thm:n-1F-ess}, Theorem \ref{thm:nF-ess},
Lemma \ref{lem:bJ-J} and Proposition \ref{prop:J-rec-inv}.
\end{proof}

\begin{proof}[Proof of Theorem \ref{thm:C-set-thm}]
Since $F$ is a C-set, there exists an idempotent $p\in h(\mathcal J)$ such that $F\in p$.
Let $x=\mathbf{1}_F\in \{0,1\}^\mZ$ and $y=px\in[1]$. Then $y$ is $\mathcal J$-recurrent,
$x$ is strongly proximal to $y$ and $N(x,[1])=F$.

Let $\{s_{\phi(\alpha)}=(s^{(1)}_\alpha,\ldots,s^{(m)}_\alpha)\}$ be an IP-system in $\mZ^m$.
Let $U_1=[1]$. Since $N((x,y),U_1\times U_1)$ is a J-set,
there exists $r_1\in\N$ and $\alpha_1\in\pn$ such that
$\sigma\times\sigma^{r_1+s^{(i)}_{\alpha_1}}(x,y)\in U_1\times U_1$ for $i=1,2,\ldots,m$.
By continuity of $\sigma$, choose a neighborhood $U_2$ of $y$ such that $U_2\subset U_1$ and
\[\bigcup_{i=1}^m \sigma^{r_1+s^{(i)}_{\alpha_1}}U_2\subset U_1.\]

Now suppose that we have choose neighborhood $U_1,U_2,\ldots U_n, U_{n+1}$ of $y$, $r_1,r_2,\ldots, r_n$ in $\N$
and $\alpha_1<\alpha_2<\cdots<\alpha_n$ in $\pn$ satisfying the following conditions.
For every $\beta\subset\{1,2,\ldots, n\}$,
let $r_\beta=\sum_{j\in\beta}r_j$, $\phi(\beta)=\bigcup_{j\in\beta}\alpha_j$ and
$U_\beta=U_{\min \beta}$, we have
\begin{enumerate}
\item $\sigma^{r_\beta+s^{(i)}_{\phi(\beta)}}x \in U_\beta$ for $i=1,\ldots,m$.
\item $\sigma^{r_\beta+s^{(i)}_{\phi(\beta)}}U_{n+1} \subset U_\beta$ for $i=1,\ldots,m$
\end{enumerate}

Since $N((x,y),U_{n+1}\times U_{n+1})$ is a J-set,
there exists $r_{n+1}\in\N$ and $\alpha_{n+1}>\alpha_n$ such that
$\sigma\times\sigma^{r_{n+1}+s^{(i)}_{\alpha_{n+1}}}(x,y)\in U_{n+1}\times U_{n+1}$ for $i=1,2,\ldots,m$.
Choose a neighborhood $U_{n+2}$ of $y$ such that $U_{n+2}\subset U_{n+1}$ and
\[\bigcup_{i=1}^m \sigma^{r_{n+1}+s^{(i)}_{\alpha_{n+1}}}U_{n+2}\subset U_{n+1}.\]

Now we show that (1) and (2) are satisfied with $\beta$ replaced by $\beta'=\beta\cup\{n+1\}$
and $n+1$ replaced by $n+2$. This in fact follows from

\[\sigma^{r_{\beta'}+s^{(i)}_{\phi(\beta')}}x \in
\sigma^{r_{\beta}+s^{(i)}_{\phi(\beta)}}(\sigma^{r_{n+1}+s^{(i)}_{\alpha_{n+1}}} x)
\in \sigma^{r_{\beta}+s^{(i)}_{\phi(\beta)}} U_{n+1}\subset U_\beta\]
and
\[\sigma^{r_{\beta'}+s^{(i)}_{\phi(\beta')}}U_{n+2}\subset
\sigma^{r_{\beta}+s^{(i)}_{\phi(\beta)}}(\sigma^{r_{n+1}+s^{(i)}_{\alpha_{n+1}}}U_{n+2})
\subset\sigma^{r_{\beta}+s^{(i)}_{\phi(\beta)}} U_{n+1}\subset U_\beta.\]
Then by induction, we have that $\sigma^{r_{\beta}+s^{(i)}_{\phi(\beta)}}x\in [1]$
for every $\beta\in\pn$ and $i=1,2,\ldots,m$.
Thus, we obtain that for every $\alpha\in\pn$, $\bar r_\alpha^{(m)}+s_{\phi(\alpha)} \in F^m$.
\end{proof}

\begin{rem}
One can use the algebraic properties of $\bN$ to prove Theorem \ref{thm:C-set-thm} (\cite{BBDF09,HS09}).
It is interesting that whether we can prove Lemma \ref{lem:n-1c-set} by algebraic properties of $\bN$.
\end{rem}

\subsection*{Acknowledgement}
The author would like to thank Prof.\@ Xiangdong Ye
for helpful suggestions which significantly improved the presentation.
This work was partly supported  by the National Natural Science Foundation of China (Nos 11171320, 11001071, 11071231).

\end{document}